\documentclass[12pt,reqno]{amsart}
\usepackage{amssymb}
\usepackage{graphicx}
\usepackage{amsmath}
\usepackage{a4wide}
\usepackage{version}
\usepackage{mathrsfs}  
\usepackage{tikz}
\usepackage{mathtools}
\usetikzlibrary{arrows,decorations.markings, matrix}
\usepackage{hyperref}

\usepackage[all]{xy}

\usepackage[small,nohug]{diagrams}
\diagramstyle[labelstyle=\scriptsize] 

\numberwithin{equation}{section}

\newtheorem{thm}{Theorem}[subsection]
\newtheorem{lem}[thm]{Lemma}
\newtheorem{prop}[thm]{Proposition}
\newtheorem{cor}[thm]{Corollary}

\theoremstyle{definition}

\newtheorem{example}[thm]{Example}
\newtheorem{rmk}[thm]{Remark}

\newcommand{\Z}{\mathbb{Z}} 
\renewcommand{\k}{\mathbf{k}}

\newcommand{\cA}{\mathcal{A}}
\newcommand{\Exc}{{\bf Exc}}

\newcommand{\Pf}{\noindent {\it Proof}}
\newcommand{\id}{\operatorname{id}}

\newcommand{\ov}{\overline}

\newcommand{\FF}{{\mathcal F}}
\newcommand{\WW}{{\mathcal W}}
\newcommand{\EE}{{\mathcal E}}
\newcommand{\MM}{{\mathcal M}}
\newcommand{\TT}{{\mathcal T}}

\newcommand{\PP}{{\mathcal P}}

\renewcommand{\SS}{{\mathcal S}}
\newcommand{\LL}{{\mathcal L}}

\newcommand{\Hom}{\operatorname{Hom}}

\newcommand{\Ext}{\operatorname{Ext}}

\newcommand{\Aut}{\operatorname{Aut}}

\newcommand{\la}{\lambda}

\newcommand{\C}{{\mathbb C}}

\newcommand{\wt}{\widetilde}
\newcommand{\ot}{\otimes}

\newcommand{\sub}{\subset}
\newcommand{\ed}{\qed\vspace{3mm}}

\newcommand{\Qcoh}{\operatorname{Qcoh}}

\renewcommand{\k}{\mathbf{k}}

\newcommand{\bC}{{\mathcal C}}

\renewcommand{\mod}{\operatorname{mod}}

\newcommand{\und}{\underline}

\newcommand{\OO}{{\mathcal O}}

\newcommand{\DD}{{\mathcal D}}

\newcommand{\II}{{\mathcal I}}

\newcommand{\hra}{\hookrightarrow}
\newcommand{\lan}{\langle}
\newcommand{\ran}{\rangle}
\newcommand{\Coh}{\operatorname{Coh}}
\newcommand{\GG}{{\mathcal G}}
\newcommand{\CC}{{\mathcal C}}

\renewcommand{\P}{{\mathbb P}}

\newcommand{\si}{\sigma}

\newcommand{\Perf}{\operatorname{Perf}}

\newcommand{\ba}{{\bf a}}

\title[Auslander orders and partially wrapped Fukaya categories]{Auslander orders over nodal stacky curves and partially wrapped Fukaya categories}
\author{Yank\i\ Lekili} 
\author{Alexander Polishchuk} 

\address{King's College London}
\address{University of Oregon and National Research University Higher School of Economics}

\begin{document} 

\begin{abstract} It follows from the work of Burban and Drozd \cite{BD} that for nodal curves
    $C$, the derived category of modules over the Auslander order
    $\mathcal{A}_C$ provides a categorical (smooth and proper) resolution of
    the category of perfect complexes $\mathrm{Perf}(C)$. On the A-side, it follows from the work of
    Haiden-Katzarkov-Kontsevich \cite{HKK} that for punctured surfaces $X$ with stops $\Lambda$ at their boundary, the partially wrapped Fukaya
    category $\mathcal{W}(X,\Lambda)$ provides a categorical (smooth and proper) resolution of the compact Fukaya
    category $\mathcal{F}(X)$. Inspired by this analogy, we establish an equivalence between the derived category of
modules over the Auslander orders over certain nodal stacky curves and partially wrapped Fukaya categories
associated to punctured surfaces of arbitrary genus equipped with stops at their boundary.
    As an application, we deduce equivalences between derived categories of coherent sheaves
    (resp. perfect complexes) on
    such nodal stacky curves and the wrapped (resp. compact) Fukaya categories of punctured surfaces
    of arbitrary genus.
\end{abstract}

\maketitle

\section*{Introduction}

Let $X$ be a Liouville domain. There are two flavours of Fukaya categories (defined over $\mathbb{Z}$)
that one can associate to
$X$: 
\begin{align*}
    &\mathcal{F}(X): \text{\ the (split-closed) derived Fukaya category of compact exact Lagrangians in\ }X, \\ 
    &\mathcal{W}(X): \text{\ the (split-closed) derived wrapped Fukaya category of\ }X. 
\end{align*}
(Here, we are suppressing the extra choices of grading structures on $X$ and brane structures on
the objects).

By construction, $\mathcal{F}(X)$ embeds as a full
subcategory of $\mathcal{W}(X)$, but there are also additional objects in $\mathcal{W}(X)$
corresponding to non-compact Lagrangians in $X$. 

On the other hand, given a scheme (or an algebraic stack) $\mathcal{C}$, one can associate two
pre-triangulated DG-categories: 
\begin{align*}
    &\mathrm{Perf}(\mathcal{C}): \text{\ the (DG-)derived category of perfect complexes on\ }
    \mathcal{C},\\
    &D^b (\Coh \mathcal{C}): \text{\ the (DG-)derived category of coherent sheaves
    on\ } \mathcal{C},
\end{align*}

In a previous work \cite{LP} (cf. \cite{LPer}), the authors studied these categories for $X =
T_d$ a $d$-holed
torus, and $\mathcal{C} = \mathcal{C}_d$ the standard (N\'eron) $d$-gon (For $d=1$, $\mathcal{C}_1$ is the nodal projective cubic
$\{y^2 z +xyz= x^3 \}$ in $\mathbb{P}^2_\mathbb{Z}$), and proved the homological 
mirror symmetry statement that identifies the following triangulated categories, all defined over $\mathbb{Z}$: 
\begin{align}\label{Tn-mirror-eq} 
    &\mathcal{F}(T_d) \simeq \mathrm{Perf}(\mathcal{C}_d) \\ \label{Tn-mirror-bis-eq}
    &\mathcal{W}(T_d) \simeq D^b(\Coh \mathcal{C}_d).
\end{align}

In recent years, a theory of partially wrapped Fukaya categories was developed
(\cite{auroux},\cite{sylvan}, \cite{HKK}, \cite{EkLe}). This depends on
an extra choice of a Legendrian submanifold $\Lambda$ at the boundary of $X$. 
In the case that $X$ is (real) 2-dimensional, $\Lambda$ is simply determined by picking boundary marked
points. In \cite{HKK}, Haiden-Katzarkov-Kontsevich gave a combinatorial description of
the resulting partially wrapped Fukaya categories when $X$ is a (real) 2-dimensional
symplectic manifold with non-empty boundary and a choice of $\Lambda$ in its boundary. We
will denote such \emph{partially wrapped Fukaya categories} by
\[ \mathcal{W}(X, \Lambda) \]
since taking $\Lambda$ to be empty gives the wrapped Fukaya category $\mathcal{W}(X)$.

In view of equivalence \eqref{Tn-mirror-bis-eq}, involving the wrapped Fukaya category
$\mathcal{W}(T_d)$, it is natural to
wonder about the homological mirror symmetry for the categories $\mathcal{W}(X, \Lambda)$ when
$X=T_d$ and $\Lambda$ is a
number of points on each boundary component of $T_d$. We specify the choice of such points by a
$d$-tuple of integers $(m_1,m_2,\ldots, m_d)$. In dimension two, it follows by construction
that $\mathcal{W}(T_n, \Lambda)$ depends only on the $d$-tuple $(m_1,m_2,\ldots, m_d)$. More
generally, for arbitrary $g\ge 0$ we denote the partially wrapped Fukaya categories of the $d$-holed
genus $g$ surface with $(m_1,m_2,\ldots,m_d)$ marked points on its boundary by
\[ \mathcal{W}(g;m_1,m_2,\ldots, m_d). \]
Up to equivalence (and choice of grading structures), 
these are all the partially wrapped Fukaya categories in dimension two. 
When there is a repetition on $m_i$, say, $m_{j+1}=m_{j+2}= \ldots =m_{j+r}=m$, 
we will use the abbreviation $(m)^{r}$ instead of writing $r$ consecutive $m$'s. 

On the A-side, we describe our surfaces as built by taking a sequence of numbers
$(r_0,r_1,\ldots,r_n)$ (resp. $(r_1,r_2,\ldots,r_n)$) and considering $n$ annuli with
$(r_i,r_{i+1})$ marked points, and connecting them via strips so as to form a chain (resp. a
ring) of annuli. The way that the two neighbouring annuli are connected via strips are
encoded by permutations $\sigma_{i} \in \mathfrak{S}_{r_i}$. We find a generating set
of Lagrangians of the partially wrapped Fukaya category adapted to this description and follow the combinatorial description given in
\cite{HKK} to explicitly compute a quiver algebra representing these categories. The permutations
$\{ \sigma_{i} \} $ play a crucial role to access higher genus surfaces - if all of them were taken
to be identity, then one would only get genus 0 or 1 surfaces. 

On the B-side we consider categorical resolutions of the perfect derived
categories of certain nodal stacky curves.
The nodal stacky curves in question are slight generalizations of the balloon chains and rings introduced in  \cite{STZ}.
We generalize to such curves the construction of Auslander orders given in the work of Burban and Drozd
\cite{BD} for the usual nodal curves. These categories of modules over the          Auslander orders
turn out to match partially wrapped Fukaya categories $\WW(g;m_1,m_2,\ldots,m_d)$ for appropriate $g$ and $m_i$.

Recall (see \cite{STZ}) that a {\it balloon} $B(a,b)$, for $a,b >0$, is a weighted projective line with
two stacky points $q_-$ and $q_+$ such that $\Aut(q_-)=\mu_a$, $\Aut(q_+)=\mu_b$.
The {\it balloon chain} (resp., {\it balloon ring} $R(r_1,\ldots,r_n)$)
is the union of balloons $B(r_0,r_1), \ldots, B(r_{n-1},r_n)$ (resp.,
$B(r_1,r_2), \ldots, B(r_{n-1},r_n), B(r_n,r_1)$) glued along
their stacky points so that they form a chain (resp., ring). It is also required that
every node locally looks like the quotient of $xy=0$ by
the action of $\mu_r$ of the form $\zeta\cdot (x,y)=(\zeta^k x,\zeta y)$ for some $k\in(\Z/r\Z)^*$.
Note that in \cite{STZ} only balanced stacky nodes were allowed (those with $k=-1$). As we will see
below, the extension to non-balanced case is crucial to construct mirrors to punctured surfaces
of genus $g>1$.  

We denote the above balloon chain (resp., ring) by $C(r_0,\ldots,r_n;k_1,\ldots,k_{n-1})$ 
(resp., $R(r_1,\ldots,r_n;k_1,\ldots,k_n)$), where $k_i\in(\Z/r\Z)^*$ describe the type of the stacky node
connecting $B(r_{i-1},r_i)$ and $B(r_i,r_{i+1})$.

In the case of $C(r_0,\ldots,r_n;k_1,\ldots,k_{n-1})$ we also allow the possibility for $r_0=0$ (resp., $r_n=0$): in
this case $B(0,r_1)$ (resp. $B(r_{n-1},0)$) denotes
the weighted affine line $\mathbb{A}^1(r_1)=B(1,r_1)\setminus\{q_-\}$ (resp. 
$\mathbb{A}^1(r_{n-1})=B(r_{n-1},1)\setminus\{q_+\}$). 

The {\it Auslander order} over a reduced curve $C$ is defined in \cite{BD} by the   formula
\begin{equation}\label{aus-def-eq}
    \cA_C=\left(\begin{matrix} \wt{\OO} & \II \\ \wt{\OO} & \OO_C
\end{matrix}\right),
\end{equation}
where $\wt{\OO}$ is the push-forward of the structure sheaf of the normalization of $C$ and
$\II\sub\OO_C$ is the ideal sheaf of the singular locus. 
We apply the similar definition to our stacky curves $C(r_0,\ldots,r_n;k_1,\ldots,k_{n-1})$ and
$R(r_1,\ldots,r_n;k_1,\ldots,k_n)$.

Now we can state our first main result. Let us work over a field $\k$.

\medskip

\noindent
{\bf Theorem A}. {\it For $\bC=C(r_0,\ldots,r_n;k_1,\ldots,k_{n-1})$ with $r_0, r_n \geq 0$ and $r_i \geq 1$ for
$i=1,\ldots, n-1$, we have an equivalence
\[ D^b (\cA_\bC-\mod) \simeq \WW(g;r_0,(2d_1)^{p_1},\ldots,(2d_{n-1})^{p_{n-1}},r_n), \]
where $p_i=\gcd(k_i+1,r_i)$, $d_i=r_i/p_i$, and 
\begin{equation}\label{g-lin-eq}
 g= \frac{1}{2}\sum_{i=1}^{n-1}(r_i-p_i).
 \end{equation}
For $\bC=R(r_1,\ldots,r_n;k_1,\ldots,k_n)$, with $r_i \geq 1$, we have an equivalence
\[ D^b (\cA_\bC-\mod) \simeq  \WW(g;(2d_1)^{p_1},\ldots,(2d_n)^{p_n}), \]
where $d_i$ and $p_i$ are defined in the same way as before, and
\begin{equation}\label{g-circ-eq}
g= 1+\frac{1}{2}\sum_{i=1}^{n}(r_i-p_i).
 \end{equation}
In both cases the equivalence holds for certain choice of grading on the partially wrapped Fukaya category
(that depends on $k_i$'s).
}

We observe that changing the gluing along nodes by varying the value of $k_i$, mirrors the use of
permutations $\{ \sigma_i \}$ in attaching strips between cylinders (though, this only covers
certain permutations). In particular, the balanced nodes mirror attaching strips via identity permutation. Indeed, note that in the case of balanced nodes, i.e., when $k_i\equiv -1 \mod r_i$ for all $i$, 
for $\bC=C(r_0,\ldots,r_n):=C(r_0,\ldots,r_n;-1,\ldots,-1)$ 
we get an equivalence involving genus $0$ surface:
\[ D^b (\cA_\bC-\mod) \simeq \WW(0;r_0,(2)^{r_1+\ldots+r_{n-1}},r_n). \]
For $\bC=R(r_1,\ldots,r_n):=R(r_1,\ldots,r_n;-1,\ldots,-1)$ we get an equivalence involving genus $1$ surface:
\[ D^b (\cA_\bC-\mod) \simeq  \WW(1;(2)^{r_1+\ldots+r_n}). \]

On the A-side one can also consider the \emph{infinitesimally wrapped Fukaya category} (cf. \cite{NZ})
\[ \mathcal{F}(X, \Lambda). \] 
We have functors
\[ \mathcal{F}(X) \to \mathcal{F}(X,\Lambda) \to \mathcal{W}(X,\Lambda) \to \mathcal{W}(X),\]
where the first two functors are full and faithful embeddings, and the
last one is a localization functor to the quotient
of $\mathcal{W}(X,\Lambda)$ by the full subcategory generated by objects supported near $\Lambda$
(see \cite[Sec. 3.5]{HKK}). 

Let us denote by 
$$\mathcal{F}(g;m_1,m_2,\ldots,m_d)\sub \WW (g;m_1,m_2,\ldots,m_d)$$
the full $A_\infty$-subcategory consisting of direct sums of objects corresponding to Lagrangians
that do not end on the boundary components with no marked points (i.e., such that the corresponding
$m_i=0$). The notation is chosen to emphasize that this full $A_\infty$-subcategory is the essential
image of the full and faithful functor $\mathcal{F}(X;\Lambda) \to \mathcal{W}(X;\Lambda)$ from
the infinitesimally wrapped Fukaya category to partially wrapped Fukaya category. In Section
\ref{dualities}
we also prove that there is a natural quasi-equivalence
$$\FF(g;m_1,\ldots,m_d)\xrightarrow{\ \sim\ } \mathrm{Fun^{ex}}(\WW(g;m_1,\ldots,m_d)^{op},\mathrm{Perf}\, \k).$$
where $\mathrm{Fun}^{ex}$ stands for DG-category of exact functors.

As an application of Theorem A, we deduce an equivalence of the perfect derived category of $\bC$
(resp., the derived category of coherent sheaves on $\bC$) with the appropriate infinitesimally wrapped Fukaya
category (resp., wrapped Fukaya category).
In particular, we obtain simpler proofs of mirror symmetry for
punctured surfaces of genus $0$ and $1$ (see \cite{AAEKO} and \cite[Thm.\ B]{LP}), and we also get a homological mirror symmetry result 
for all surfaces of genus $g>1$ with at least one puncture. 

\medskip

\noindent
{\bf Theorem B.} {\it For $\bC=C(r_0,\ldots,r_n;k_1,\ldots,k_{n-1})$ with $r_0, r_n \geq 0$ and $r_i \geq 1$ for
$i=1,\ldots, n-1$, 
we have equivalences (with some choice of grading on the relevant Fukaya categories)
$$D^b(\Coh \bC)\simeq \WW(g;r_0,(0)^{p_1+\ldots+p_{n-1}},r_n),$$
$$\Perf_c(\bC)\simeq \mathcal{F}(g;r_0,(0)^{p_1+\ldots+p_{n-1}},r_n),$$
where $p_i=\gcd(k_i+1,r_i)$, $g$ is given by \eqref{g-lin-eq}, and $\Perf_c(\bC)$ is the full subcategory in $\Perf(\bC)$ consisting of complexes with
proper support (the condition on support is only needed if $r_0=0$ or $r_n=0$).

For $\bC=R(r_1,\ldots,r_n;k_1,\ldots,k_n)$, with $r_i \geq 1$,  we have equivalences
$$D^b(\Coh \bC)\simeq \WW(g;(0)^{p_1+\ldots+p_n}),$$
$$\Perf(\bC)\simeq \mathcal{F}(g;(0)^{p_1+\ldots+p_n})$$
where $p_i=\gcd(k_i+1,r_i)$ and $g$ is given by \eqref{g-circ-eq}.
In particular, for any $g\ge 2$, choosing $k\in(\Z/(2g-1))^*$ such that $k+1\in (\Z/(2g-1))^*$ (such $k$ always exists), 
for every $n\ge 1$ we get equivalences
$$D^b(\Coh R(2g-1,(1)^{n-1};k))\simeq \WW(g;(0)^{n}),$$
$$\Perf(R(2g-1,(1)^{n-1};k))\simeq \FF(g;(0)^{n}).$$
}

We stress that for equivalences of Theorem A and Theorem B to hold we choose specific line fields on the surfaces
constructed from our data $(r_0,\ldots,r_n;k_1,\ldots,k_{n-1})$ (resp., $(r_1,\ldots,r_n;k_1,\ldots,k_n)$).
The Fukaya categories depend on these choices. Namely, we will see that the Fukaya categories considered will be equivalent to the derived category of modules of a certain graded algebra given as the endomorphism algebra of generating set of objects (see Section \ref{Fuk-sec}). Changing the line field would result in changing the grading of this algebra which will not in general be derived equivalent to the original algebra even though the underlying ungraded algebras are the same. Therefore, if different data lead to homeomorphic (marked) surfaces
one cannot conclude in general that the corresponding Fukaya categories are equivalent. To determine which
of these categories are equivalent one should study the action of the mapping class group on the homotopy classes
of line fields. In a follow-up paper (\cite{LPol}), we described explicitly invariants of line fields under this action.
This leads to many interesting equivalences between the corresponding categories on the B-side,
generalizing the known equivalences in the balanced case $k_i=-1$ (see \cite{sibilla}).

Note that the B-model categories that previously appeared in homological mirror
symmetry for higher genus surfaces were given in terms of matrix factorizations categories of some $3$-dimensional
Landau-Ginzburg models (cf. \cite{AAEKO}, \cite{bocklandt}, \cite{HLee}, \cite{PS}). In our picture 
the B-model categories are the usual
derived categories associated with (commutative) stacky curves. This is more in line with the traditional
homological mirror symmetry conjecture \cite{kontsevich}. In the simplest cases the relation to the $3$-dimensional
Landau-Ginzburg models can be checked purely on the B-side via Kn\"orrer periodicity.

We prove Theorem B by identifying $\Perf(\bC)$ (resp., $D^b(\Coh \bC)$) with an explicit full subcategory 
(resp., localization) of $D^b(\cA_\bC-\mod)$, generalizing similar constructions by Burban-Drozd in
the non-stacky case (see \cite[Sec.\ 3,4]{BD}).
We also show that looking at other localizations of $D^b(\cA_\bC-\mod)$ one gets categorical resolutions of the categories 
$\Perf(\bC)$ (see Proposition \ref{nc-res-prop}).

The paper is organized as follows. In Section \ref{Auslander-sec} we discuss Auslander orders on balloon chains and balloon rings
and the categories of modules over them. The main result of this Section is Theorem \ref{B-side-thm} describing full
exceptional collections on these stacky curves 
(generalizing the exceptional collection in the non-stacky case constructed in \cite{BD}).
In Section \ref{Fuk-sec} we find similar exceptional collections in the partially wrapped Fukaya
categories of punctured surfaces of arbitrary genus
and prove Theorem A. In Section \ref{local} we consider objects in the partially wrapped Fukaya category supported near
marked points of the boundary and identify the corresponding modules over Auslander orders. Finally, in 
Section \ref{Perf-sec} we identify the subcategory in the partially wrapped Fukaya categories corresponding under the equivalence
of Theorem A to the subcategory $\Perf(\bC)$, thus, proving Theorem B.

Everywhere in this paper we work over a fixed ground field $\k$ (although our descriptions of Fukaya category are also valid
over $\Z$).

\medskip

\noindent
{\it Acknowledgments}. Y.L. is partially supported by the Royal Society (URF) and
the NSF grant DMS-1509141, and would like to thank Igor Burban for discussions at an early stage. A.P. is supported in part by the NSF grant DMS-1400390 and by the Russian Academic Excellence Project `5-100'. He would like to thank Paolo Stellari
for help with locating the reference \cite{miyachi}.

\section{Modules over Auslander orders}\label{Auslander-sec}

\subsection{Burban-Drozd tilting for nodal curves}

Let $C$ be a reduced projective curve, and let $\pi:\wt{C}\to C$ be its normalization.
The Auslander order $\cA_C$ over $C$ is the order given by 
\eqref{aus-def-eq}
where $\wt{\OO}=\pi_*\OO_{\wt{C}}$ and
$\II\sub\OO_C$ is the ideal sheaf of the singular locus. 
We denote by $\cA_C-\mod$ the category of coherent left $\cA_C$-modules.

Burban and Drozd have shown in \cite{BD} that if $C$ has only nodal or cuspidal singularities
then the category $\cA_C-\mod$ has global dimension $2$. 
Furthermore, if in addition all the components of $C$ are rational then they constructed a strong exceptional
collection generating $D^b(\cA_C-\mod)$.

Let us recall the form of this exceptional collection in the case when $C$ is either a chain or a ring of $\P^1$'s joined nodally
({\it standard $n$-gon}).
Let $\pi_i:\wt{C}_i\to C$, $i=1,\ldots,n$, be the restriction of the normalization map 
to the irreducible components of $\wt{C}$, and let $C_i=\pi_i(\wt{C}_i)$. 
For $i=1,\ldots,n$ and $j\in\Z$ we define an $\cA_C$-module 
$$\PP_i(j)=\left(\begin{matrix} \pi_{i*}\OO(j)\\ \pi_{i*}\OO(j) \end{matrix}\right).
$$ 
Also, for each node $q\in C$ we have a simple $\cA_{C}$-module
$$\SS_q=\left(\begin{matrix} 0 \\ \OO_q \end{matrix}\right).$$

It is proved in \cite[Sec.\ 5]{BD} that the objects
$$(\SS_q[-1] \ |\ q \text{ is a node of } C), (\PP_i(-1),\PP_i \ |\  i=1,\ldots,n)$$
form a full strong exceptional collection and its endomorphism algebra has a simple presentation.
Namely, in the case when $C$ is a chain, and the nodes are $q_1,\ldots,q_{n-1}$, with $q_i\in C_i\cap C_{i+1}$,
the morphism spaces 
$$\Hom(\PP_i(-1),\PP_i)=\k\cdot x_i\oplus \k\cdot y_i,$$
$$\Hom(\SS_{q_i}[-1],\PP_i(-1))=\k\cdot a_i, \ \Hom(\SS_{q_i}[-1],\PP_{i+1}(-1))=\k\cdot b_i$$
generate the endomorphism algebra, with the defining relations 
$$y_ia_i=0, \ x_{i+1}b_i=0.$$
In the case when $C$ is a ring and the nodes are $q_i\in C_i\cap C_{i+1}$, where $i\in\Z/n$, the description is the same
for $n\ge 2$, with the convention that $i\in\Z/n$. In the case $n=1$, the only difference is that $a_1$ and $b_1$ are
elements of the same space $\Hom(\SS_{q_1}[-1],\PP_1(-1))$, which is now $2$-dimensional.

\subsection{Auslander orders on stacky curves}\label{Aus-st-sec}

Let $\bC$ be either a balloon chain or a balloon ring with the components $\bC_1,\ldots,\bC_n$,
glued along the stacky nodes $q_i\in \bC_i\cap \bC_{i+1}$. Note that in the case
of a balloon ring we view the index $i$ as an element of $\Z/n$, whereas in the case of a balloon chain the nodes
are $q_1,\ldots,q_{n-1}$.

We have a natural morphism $\pi:\wt{\bC}\to \bC$ from the disjoint union of the stacky projective lines and we
set $\wt{\OO}=\pi_*\OO_{\wt{\bC}}$. We denote by $\II\sub\OO_{\bC}$ the ideal sheaf of the union of the nodes.
Then the Auslander order $\cA_{\bC}$ over $\bC$ is defined by the same formula \eqref{aus-def-eq}.
Let us define an $\cA_{\bC}$-module by 
$$\FF_\bC:=\left(\begin{matrix} \OO\\ \II \end{matrix}\right).$$
This module will play an important role later in connecting the category $D^b(\cA_\bC-\mod)$ with $\Perf(\bC)$ and
$D^b(\Coh \bC)$ (see Propositions \ref{DbCoh-prop} and \ref{Perf-char-B-prop}).

Balloons $B(a,b)$ (where $a,b>0)$ are examples weighted projective lines of \cite{GL}. The derived category of coherent sheaves on a balloon $B(a,b)$ (where $a,b>0$) is generated by the exceptional collection of line bundles (see \cite[Sec. 4]{GL}):

\begin{equation}
\begin{tikzpicture}[baseline=-0.5ex]
\label{balloon-exc-coll}
\matrix (m) [matrix of math nodes,row sep=3em,column sep=4em,minimum width=2em]
  {
  \OO(-aq_-)& \OO(-(a-1)q_-)&\cdots&\OO(-q_-)&\OO\\
  \OO(-bq_+)& \OO(-(b-1)q_+)&\cdots&\OO(-q_+)&\OO\\ };
      
 \path[-stealth]
    (m-1-1) edge node [left] {\tiny $=$} (m-2-1)
            edge node [above] {\tiny $x(-a+1)$} (m-1-2)
    (m-1-2) edge node [above] {\tiny $x(-a+2)$} (m-1-3)
    (m-1-3) edge node [above] {\tiny $x(-1)$} (m-1-4)
    (m-1-4) edge node [above] {\tiny $x(0)$} (m-1-5)

    (m-1-5) edge node [left] {\tiny $=$} (m-2-5)

    (m-2-1) edge node [above] {\tiny $y(-b+1)$} (m-2-2)
    (m-2-2) edge node [above] {\tiny $y(-b+2)$} (m-2-3)
    (m-2-3) edge node [above] {\tiny $y(-1)$} (m-2-4)
    (m-2-4) edge node [above] {\tiny $y(0)$} (m-2-5);

\end{tikzpicture}
\end{equation}

Here the endomorphism algebra of this collection is simply the path algebra of the corresponding quiver.

In the rest of this section we assume that all $r_i$ are positive (i.e., we do not allow a balloon chain with $r_0=0$ or $r_n=0$). 
Let $\wt{\bC}_i$ be the components of $\wt{\bC}$ and let $q_{i,-},q_{i,+}\in \wt{\bC}_i$ be the stacky points, 
with $|\Aut(q_{i,\pm})|=r_{i,\pm}$, so that $q_{i,+}$ and $q_{i+1,-}$
get glued into the node $q_i$ in $\bC$ (so $r_{i+1,-}=r_{i,+}$), and let $\pi_i:\wt{\bC}_i\to\bC$ be the natural projection.
Then we have 
$$\II=\bigoplus_i \pi_{i*}\OO(-q_{i,-}-q_{i,+}).$$

For integers $j$ and $m$ we set 
$$\PP_i(j,m)=\left(\begin{matrix} \pi_{i*}\OO(jq_{i,-}+mq_{i,+})\\ \pi_{i*}\OO(jq_{i,-}+mq_{i,+}) \end{matrix}\right),$$
Note that $\PP_i(j+r_{i,-},m)\simeq \PP_i(j,m+r_{i,+})$. We have a decomposition of left $\cA_\bC$-modules
\begin{equation}\label{A-decomp-eq}
\cA_\bC\simeq \FF_\bC\oplus\bigoplus_{i=1}^n \PP_i(0,0).
\end{equation}

We claim that $\PP_i(j,m)$ are exceptional objects in $D^b(\cA_\bC-\mod)$ and that by restricting $(j,m)$ appropriately
we get exceptional collections with the same endomorphism algebra 
as the corresponding exceptional collection on the balloon $\bC_i$.

\begin{lem}\label{P-exc-lem} 
(i) The $\cA_{\bC}$-modules $\PP_i(j,m)$ are exceptional.

\noindent
(ii) For $i\neq i'$ one has $\Hom^*(\PP_i(j,m),\PP_{i'}(j',m'))=0$.

\noindent
(iii) For any $(j,m)$ we get an exceptional collection 

\begin{equation}
\begin{tikzpicture}[baseline=-0.5ex]
\label{balloon-exc-coll-bis}
\matrix (m) [matrix of math nodes,row sep=3em,column sep=2.8em,minimum width=2em]
   {\PP_i(j,m)&\PP_i(j+1,m)&\cdots&\PP_i(j+r_{i,-}-1,m)&\PP_i(j+r_{i,-},m)\\
    \PP_i(j,m)&\PP_i(j,m+1)&\cdots&\PP_i(j,m+r_{i,+}-1)&\PP_i(j,m+r_{i,+})\\};
      
 \path[-stealth]
    (m-1-1) edge node [left] {\tiny $=$} (m-2-1)
            edge node [above] {\tiny $x(0)$} (m-1-2)
    (m-1-2) edge node [above] {\tiny $x(1)$} (m-1-3)
    (m-1-3) edge node [above] {} (m-1-4)
    (m-1-4) edge node [above] {\tiny $x(r_{i,-}-1)$} (m-1-5)

    (m-1-5) edge node [left] {\tiny $=$} (m-2-5)

    (m-2-1) edge node [above] {\tiny $y(0)$} (m-2-2)
    (m-2-2) edge node [above] {\tiny $y(1)$} (m-2-3)
    (m-2-3) edge node [above] {} (m-2-4)
    (m-2-4) edge node [above] {\tiny $y(r_{i,+}-1)$} (m-2-5);

\end{tikzpicture}
\end{equation}
with the same endomorphism algebra as for the exceptional collection \eqref{balloon-exc-coll}.
We denote this exceptional collection as $\Exc_i(j,m)$.
\end{lem}

\Pf . First, let us calculate $\und{\Hom}(\PP_i(j,m),\PP_{i'}(j',m'))$. Note that this is a local
calculation, so near the nodes we can use the presentation as a quotient of $xy=0$ by $\mu_r$.
Thus, using the similar calculation in the non-stacky case (see \cite[Cor.\ 5.5]{BD}), we derive
that the above $\und{\Hom}$ vanishes for $i\neq i'$, while for $i=i'$ we get
$$\und{\Hom}(\PP_i(j,m),\PP_i(j',m'))\simeq \pi_{i*}L$$
for appropriate line bundle $L$ on $\bC_i$.
Thus, we are reduced to a calculation of cohomology on the balloon $\bC_i$, i.e., to
the standard exceptional collection \eqref{balloon-exc-coll} on the balloon curve twisted by a line bundle.
\ed

As in the non-stacky case, for each node $q\in\bC$ we have a simple $\cA_{\bC}$-module
$$\SS_q=\left(\begin{matrix} 0 \\ \OO_q \end{matrix}\right).$$

Recall that we assume that locally near each node we can identify $\bC$ with the quotient of $xy=0$ by the action of 
$\mu_r$ of the form $\zeta\cdot (x,y)=(\zeta^kx,\zeta y)$ for some $k\in(\Z/r)^*$.
Using this identification, 
locally we can view $\cA_{\bC}$-modules as $\mu_r$-equivariant modules over the Auslander order on $xy=0$.
Thus, if we fix an identification $\Aut(q)=\mu_r$ then 
for every character $x\mapsto x^c$ of $\mu_r$ we have a twist operation $M\mapsto M\{c\}$ on
$\cA_{\bC}$-modules supported at the node $q$. 
For our stacky curves we fix identifications $\Aut(q_i)=\mu_{r_i}$, where 
$$r_i=r_{i,+}=r_{i+1,-},$$ 
in such a way that
$\mu_{r_i}$ acts on the fiber of $\OO(-q_{i+1,-})$ at $q_{i+1,-}$ through its natural character
(i.e., $\zeta\in\mu_{r_i}$ acts by mutliplication with $\zeta$). Then there exists a unique $k_i\in(\Z/r_i)^*$ such that
$\mu_{r_i}$ acts on the fiber of $\OO(-q_{i,+})$ at $q_{i,+}$ through the character $\zeta\mapsto \zeta^{k_i}$.
We include the parameters $(k_i)$ in the notation by writing
$\CC=C(r_0,r_1,\ldots,r_n;k_1,\ldots,k_{n-1})$ (in the case of a balloon chain), or
$\CC=R(r_1,\ldots,r_n;k_1,\ldots,k_n)$ (in the case of a balloon ring).
In the case of non-stacky nodes, i.e., when $r_i=1$, we will either write $k_i=0$ or omit $k_i$ altogether.

Let 
$$p:\bC\to C$$
be the coarse moduli for $\bC$. Note that $C$ is either a chain or a ring of projective lines.

Let us say that a quasicoherent sheaf $\EE$ on $\bC$ is a generator of $\Qcoh(\bC)$ 
with respect to $p$ if for every quasicoherent sheaf $\GG$ on $\bC$ the natural map
$$p^*p_*\und{\Hom}(\EE,\GG)\ot \EE\to \GG$$
is surjective (see \cite[Sec.\ 5]{OS}). 

Assume $\bC$ is a balloon chain.
For each collection of integers $\ba=(a_{i,\pm})$, where $0\le a_{i,\pm}<r_{i,\pm}$ and $a_{i,+}=a_{i+1,-}$, we define
a line bundle $\MM\{\ba\}$ on $\bC$
by gluing the line bundles $\OO_{\wt{\bC}_i}(k_{i-1}a_{i,-}q_{i,-}+a_{i,+}q_{i,+})$ on $\wt{\bC}_i$
(note that this gluing is well defined because the automorphism
group of the node acts on the fibers with the same character).

In the case when $\bC$ is a balloon ring the definition of line bundles $\MM\{\ba\}$ is similar (we need as many numbers in the collection
$\ba$ as there are stacky points in $\bC$). Note that in the case $n=1$ this means that we are descending the line bundle $\OO_{\wt{\bC}}(-aq_- + aq_+)$
to $\bC$.

\begin{lem}\label{generation-lem}
The vector bundle $\bigoplus_{\ba} \MM\{\ba\}$ is a generator of $\Qcoh(\bC)$ with respect to $p$.
\end{lem}

\Pf . The question is local over $C$, so it is enough to check that this is true near the stacky points. Then
we can use the presentation as a quotient by the action of the cyclic group and
\cite[Prop.\ 5.2]{OS}. 
\ed

\begin{thm}\label{B-side-thm}
Consider the stacky curve $\CC=C(r_0,r_1,\ldots,r_n;k_1,\ldots,k_{n-1})$ or $\CC=R(r_1,\ldots,r_n;k_1,\ldots,k_n)$,
where all $r_i>0$. 
For each $i=1,\ldots,n$, let us fix a pair of integers $(j_i,m_i)$.
In the case when $\bC$ is a balloon ring we view indices $i$ as elements of $\Z/n\Z$.
Then the category $D^b(\cA_{\bC}-\mod)$ is generated as a triangulated category
by the strong exceptional collection consisting of the two types of objects:
\begin{itemize}
\item
$(\SS_q\{c\}[-1])$ where $q$ is a node with  $\Aut(q)=\mu_r$, $c\in\Z/r\Z$;
\item
for each $i=1,\ldots,n$, the objects of the exceptional collection $\Exc_i(j_i,m_i)$ (see \eqref{balloon-exc-coll-bis}).
\end{itemize}
The endomorphism algebra of this exceptional collection is generated by the 
morphisms $x_i(j)$, $y_i(m)$ within each subcollection $\Exc_i(j_i,m_i)$ (which are the same as in \eqref{balloon-exc-coll-bis}),
as well as the $1$-dimensional spaces
\begin{align*}
&\Hom(\SS_{q_i}\{-j_{i+1}-j-1\}[-1], \PP_{i+1}(j_{i+1}+j,m_{i+1}))=\k\cdot b_i(j) \ \text{ and } \\
&\Hom(\SS_{q_i}\{-k_i(m_i+m+1)\}[-1], \PP_{i}(j_i,m_i+m))=\k\cdot a_i(m).
\end{align*}
The defining relations are $ya=0$ and $xb=0$ whenever the composition is possible.
\end{thm}

\Pf . We already know that the $\cA_{\bC}$-modules $\PP_i(j,m)$ are exceptional and have calculated the relevant morphisms
between them (see Lemma \ref{P-exc-lem}). Computation of morphisms involving $\SS_q\{c\}$ can be done locally near
the node $q$. Note that near $q$ we have a presentation of our stacky curve as $U/\mu_r$, where $U$ is a neighborhood of
the node in the plane curve $xy=0$. Thus, we are reduced to the computation of $\Ext$-groups in the category of
$\mu_r$-equivariant $\cA_U$-modules. From the non-stacky case considered in \cite[Sec.\ 5]{BD} we know that the
only relevant nontrivial $\Ext$-class in the category of $\cA_U$-modules is the class of the extension
\begin{equation}\label{Sq-ex-seq-U}
0\to \left(\begin{matrix} I \\ I \end{matrix}\right)\to \left(\begin{matrix} I \\ \OO_U \end{matrix}\right)\to \SS_q\to 0
\end{equation}
where $I\sub\OO_U$ is the ideal sheaf of $q$. Furthermore, this extension gives a $\mu_r$-equivariant class.
Thus, the only nontrivial morphisms involving $\SS_q$ for $q=q_i$ are one-dimensional spaces
$\Ext^1(\SS_{q_i}, \PP_i(j,m))$ for $m\equiv -1\mod r_i$ and 
$\Ext^1(\SS_{q_i}, \PP_{i+1}(j,m))$ for $j\equiv -1\mod r_i$.
Next, to find morphisms involving $\SS_{q_i}\{c\}$, we tensor the exact sequence \eqref{Sq-ex-seq-U}
by line bundles of the form $\MM\{\ba\}$ on $\bC$. Namely, it is easy to see that
$$\SS_{q_i}\ot\MM\{\ba\}\simeq \SS_{q_i}\{-k_ia\},$$
where $a=a_{i+1,-}=a_{i,+}$.
Thus, we get nontrivial elements in 
$\Ext^1(\SS_{q_i}\{-k_ia\}, \PP_{i+1}(j,m))$ for $j\equiv -1+k_ia\mod r_i$ and
in $\Ext^1(\SS_{q_i}\{-k_ia\}, \PP_i(j,m))$ for $m\equiv -1+a\mod r_i$.
This easily implies the asserted form
of the endomorphism algebra of our collection
(one has to use the fact that
the morphisms $x_i(j)$ (resp., $y_i(m)$) are isomorphisms near $q_{i,+}$ (resp., $q_{i,-}$)).

Let $\DD\sub D^b(\cA_{\bC}-\mod)$ be the triangulated subcategory generated by our exceptional collection.
The fact that the exceptional collection \eqref{balloon-exc-coll} on each balloon is full implies that for every
coherent sheaf $\GG$ on $\wt{\bC}_i$ we have
\begin{equation}\label{generation-aux-eq}
\left(\begin{matrix} \pi_{i*}\GG\\ \pi_{i*}\GG \end{matrix}\right)\in \DD.
\end{equation}
This immediately implies that $\DD$ is closed under tensoring operation
$M\mapsto M\ot \LL$ on $\cA_{\bC}$-modules, where $\LL$ is any line bundle on $\bC$.
Indeed, the objects $\PP_i(j,m)\ot\LL$ have the form as in \eqref{generation-aux-eq}, whereas
$\SS_q\{c\}\ot\LL$ is isomorphic to $\SS_q\{c'\}$ for some $c'$.

Also, \eqref{generation-aux-eq} implies that $\left(\begin{matrix} \II \\ \II \end{matrix}\right)\in \DD$. Now the exact 
sequence 
$$0\to \left(\begin{matrix} \II \\ \II \end{matrix}\right)\to \left(\begin{matrix} \II \\ \OO_{\bC} \end{matrix}\right)\to
\bigoplus_q \SS_q\to 0$$
shows that $\FF_\bC=\left(\begin{matrix} \II \\ \OO_{\bC} \end{matrix}\right)\in\DD$. Hence, 
by \eqref{A-decomp-eq}, we derive that $\cA_{\bC}\in\DD$.

Now let $L$ be an ample line bundle on $C$. Then Lemma \ref{generation-lem} implies that
the line bundles $\MM\{\ba\}\ot p^*L^m$, where $m\in\Z$, are generators for $\Qcoh(\bC)$ 
(in the sense that the orthogonal is zero; cf. the proof of \cite[Thm.\ 5.10]{BD}). It follows that
the $\cA_{\bC}$-modules $\cA_{\bC}\ot \MM\{\ba\}\ot p^*L^m$ are generators for $\Qcoh(\cA_{\bC})$.
Since all these objects are in $\DD$, this finishes the proof that our exceptional collection is full.
\ed

\section{Explicit computations of partially wrapped Fukaya categories}\label{Fuk-sec}

We will next describe several partially wrapped Fukaya categories explicitly by exhibiting
generating sets of objects and the endomorphism algebras of these objects. The combinatorial
description provided in \cite{HKK} implies that if $X$ is a surface with non-empty boundary and
$\Lambda$ is a choice of marked points at its boundary, then a set of pairwise disjoint and
non-isotopic Lagrangians $\{ L_i \}$ in $X \backslash \Lambda$ generates the
partially wrapped Fukaya category $\mathcal{W}(X;\Lambda)$ as a triangulated category
if the complement of the Lagrangians $X \setminus \{ \bigsqcup_i L_i \}$ is a union of disks each of which
has exactly one marked point on its boundary. Furthermore, in this case, the algebra \[ \bigoplus_{i,j} \mathrm{hom}(L_i, L_j) \] is formal, and it 
can be described by a graded quiver with quadratic monomial relations. The generators of
this quiver can easily be described following the flow lines corresponding to rotation around the boundary components of
$X$ connecting the Lagrangians. Note that each boundary component of $X$ is an oriented circle (where the orientation is induced by the area form on $X$).
The data of $\Lambda$ enters by disallowing flows that pass through a marked point. 
The algebra structure is given by concatenation of flow lines. Finally, we need to prescribe a
choice of a grading structure. A general definition of assigning gradings is explained in
detail in \cite[Sec. 2.1]{HKK}. The extra structure needed to define a grading is a section of the projectivized tangent bundle of $X$, which we view as an unoriented line field $l \subset TX$. Such line fields form a torsor for $C^\infty(X, \mathbb{R}P^1)$ and the connected components can be identified with $H^1(X;\mathbb{Z})$. In practice, for example when one uses a generating set of objects $\{ L_i\}$ as above,  one could apply the recipe from \cite[Sec.\ 3.2]{HKK}: if a set of generators
$x_1,\ldots,x_n$ bound a disk then one must have $\sum |x_i| = n-2$, and if a surface is glued
together from disks, choosing gradings compatible with these constraints for each disk defines a
global grading structure. Since our surfaces are glued together from
disks which have at least one marked point along the boundary, the above constraint never arises
when one looks at morphisms between $\{L_i \}$ only, so
we deduce that the gradings for (primitive) arrows on the associated quiver can be assigned
arbitrarily. We will choose a grading so that all of the arrows in the quiver have degree 0. This choice corresponds to the line field which is homotopic to the constant line field on the page everywhere in the pictures below.

Finally, we note that given $(X,\Lambda, l)$ and $(X', \Lambda', l')$, any homeomorphism $f: X \to X'$ mapping $\Lambda$ to $\Lambda'$ bijectively and such that $f_* l$ is homotopic to $l'$, induces a triangulated equivalence of corresponding partially wrapped Fukaya categories \cite[Prop. 3.2]{HKK}.  

\subsection{Computation of $\mathcal{W}(0;m)$ and $\mathcal{W}(0;m_1,m_2)$} 
\label{diskan}

We begin with two simple cases, which are well known (\cite{HKK}, \cite{STZ}).

In Figure \ref{fig1} we have a genus 0 surface with 1 boundary component, in other words, a disk
$\mathbb{D}^2$, together with $m$ marked points on its boundary. Furthermore, we depicted $m$
objects $L_1, L_2,\ldots, L_{m-1}$ from $\mathcal{W}(0;m)$. These objects do not intersect at the
interior of $\mathbb{D}^2$, thus the only morphisms between them are given by flow lines along the
boundary of $\mathbb{D}^2$. However, the marked points on the boundary serve as stops, hence the
endomorphism algebra of the object $L=L_1 \oplus L_2 \oplus \ldots L_{m-1}$ is given by the
$A_{m-1}$ quiver as in Figure \ref{fig2} with relations $a_{i+1} a_i =0$ for $i=1,\ldots, m-2$. We
grade the Lagrangians so that all the morphisms have degree $|a_i|=0$ for $i=1,2,\ldots, m-2$.

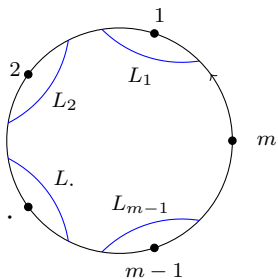
\begin{figure}[htb!]
\centering
\begin{tikzpicture}
\tikzset{vertex/.style = {style=circle,draw, fill,  minimum size = 2pt,inner        sep=1pt}}
\def \radius {1.5cm}
\tikzset{->-/.style={decoration={ markings,
        mark=at position #1 with {\arrow{>}}},postaction={decorate}}}

    \foreach \s in {1,2,3,4} {
\node[vertex] at ({360/5 * (\s)}:\radius) {} ;

    \draw ({360/5 * (\s)}:\radius) arc ({360/5 *(\s)}:{360/5*(\s+1)}:\radius);
    \draw[blue] ({360/5 * (\s) - 360*3/40}:\radius) arc ({360/5
    *(\s)+360*3/40}:{360/5*(\s)-360*3/40}:-\radius);

}

    \draw[->-=.5] ({0}:\radius) arc ({360/5 *0}:{360/5}:\radius);
 \node[vertex] at ({360/5 * 5}:\radius) {} ;

    \node[xshift=13] at ({0}:\radius) {\tiny $m$} ;

   \node[yshift=7, xshift=2] at ({360/5 }:\radius)  {\tiny $1$} ;
    \node[xshift=-5,yshift=2] at ({360/5 * (2)}:\radius)  {\tiny $2$} ;
    \node[xshift=-7, yshift=-4] at ({360/5 * (3)}:\radius)  {$\cdot$} ;
    \node[yshift=-9] at ({360/5 * (4)}:\radius)  {\tiny $m-1$} ;

    \node at ({360/5}:\radius-0.6cm) {\tiny $L_{1}$};
    \node at ({360/5 *2}:\radius-0.6cm) {\tiny $L_2$};
    \node at ({360/5 *3}:\radius-0.6cm) {{\tiny $L$}$_\cdot$ };
    \node at ({360/5 *4}:\radius-0.6cm) {\tiny $L_{m-1}$};

\end{tikzpicture}
    \caption{Objects in the partially wrapped category of $\mathbb{D}^2$
    with $m$ marked points. }
    \label{fig1}
\end{figure}

\begin{figure}[htb!]
\centering

\begin{tikzpicture}
    \tikzset{vertex/.style = {style=circle,draw, fill,  minimum size = 2pt,inner    sep=1pt}}
\tikzset{edge/.style = {->,-stealth',shorten >=8pt, shorten <=8pt  }}

\node[vertex] (a) at  (0,0) {};
\node[vertex] (c) at  (8,0) {};
\node[vertex] (x) at  (6.5,0) {};
\node[vertex] (a1) at (1.5,0) {};
\node[vertex] (a2) at (3,0) {};

\node at  (0,0.3) {\tiny 1};
\node at  (8,0.3) {\tiny $m-1$};
\node at (1.5,0.3) {\tiny 2};
\node at (3,0.3) {\tiny 3};


\draw[edge] (a)  to (a1);
\draw[edge] (a1) to (a2);

\path (a2) to node {\dots} (c);
\node [shape=circle,minimum size=2pt, inner sep=1pt] (a3) at (4.5,0) {};
\draw[edge] (a2) to (a3);

\node [shape=circle,minimum size=2pt, inner sep=1pt] (c1) at (6.5,0) {};
\draw[edge] (c1) to (c);

\node at (0.75,-0.3) {\tiny $a_1$};
\node at (7.25,-0.3) {\tiny $a_{m-2}$};
\node at (2.25,-0.3) {\tiny $a_2$};
\node at (3.75,-0.3) {\tiny $a_3$};

\end{tikzpicture}
    \caption{$A_{m-1}$ quiver, $a_{i+1} a_i=0$ for $i=1,2,\ldots,m-2$.}
\label{fig2}
\end{figure}
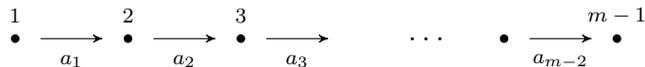

A useful observation given in \cite[Sec. 3.3]{HKK} is that we do not need to include the object
$L_m$ which is supported near
the marked point $m$, since $L_m $ is quasi-isomorphic to the twisted complex:
\begin{equation}\label{twcx} L_1[m-2] \to L_2[m-3] \to L_3[m-4] \to \ldots \to L_{m-1}
\end{equation} 

Futhermore, in fact, $L_1,L_2,\ldots,L_{m-1}$ generate the partially wrapped Fukaya category
$\mathcal{W}(0;m)$ since the union of $L_1,\ldots, L_{m-1}$ cuts
$\mathbb{D}^2$ into disks each of which has exactly one marked point.  

Next, we give an explicit presentation of the category $\mathcal{W}(0;m_1,m_2)$. In Figure
\ref{fig3} we have a genus 0 surface with 2 boundary components, with $m_1$ marked points on the
inner circular boundary component and $m_2$ marked points on the outer circular boundary
component. We also depicted $m_1+m_2$ objects, which are labeled $P^+_{0},\ldots,
P^+_{m_1}$ and $P^-_{0},\ldots,P_{m_2}^-$. For notational convenience, we have the equalities
$P_0^+ = P_0^-$ and $P_{m_1}^+ =P_{m_2}^-$. Again, by \cite[Lem. 3.3]{HKK}, since the
complement of these objects consists of disks each of which has exactly one marked point, these
objects generate the category $\mathcal{W}(0;m_1,m_2)$. 

\begin{figure}[ht!]
\centering
\begin{tikzpicture}
\tikzset{vertex/.style = {style=circle,draw, fill,  minimum size = 2pt,inner        sep=1pt}}
\tikzset{->-/.style={decoration={ markings,
        mark=at position #1 with {\arrow{>}}},postaction={decorate}}} 
\def \radius {1.5cm}

\foreach \s in {1,2,4,5,7,8,9,10} {
    \draw({360/10 * (\s)}:\radius-1cm) arc ({360/10 *(\s)}:{360/10*(\s+1)}:\radius-1cm);
         \draw ({360/10 * (\s)}:\radius+1cm) arc ({360/10 *(\s)}:{360/10*(\s+1)}:\radius+1cm);
}

\draw({360/10 * (3)}:\radius-1cm) arc ({360/10 *(3)}:{360/10*(4)}:\radius-1cm);    \draw({360/10 *
    (6)}:\radius+1cm) arc ({360/10 *(6)}:{360/10*(7)}:\radius+1cm);

    \draw[->-=.1]({360/10 * (7)}:\radius-1cm) arc ({360/10 *(7)}:{360/10*(6)}:\radius-1cm);

    \draw[->-=.1]({360/10 * (3)}:\radius+1cm) arc ({360/10 *(3)}:{360/10*(4)}:\radius+1cm);

    \foreach \s in {1,3,5} {
\node[vertex] at ({360/10 * (\s)}:\radius-1cm) {} ;

    \draw[blue] ({360/10 * ((\s)+1)}:\radius-1cm) to ({360/10 * ((\s)+1)}:\radius+1cm) ; 
}

  \foreach \s in {7,8,9,10} {
\node[vertex] at ({360/10 * (\s)}:\radius+1cm) {} ;
    \draw[blue] ({360/10 * ((\s)+0.5)}:\radius-1cm) to ({360/10 * ((\s)+0.5)}:\radius+1cm) ;

}

\node at ({360/10 *2}: \radius+1.4cm) {$P_1^+$};
    \node at ({360/10 *4}: \radius+1.4cm) {$P^+_{\cdot}$};
    \node[xshift=-0.9cm] at ({360/10 *6}: \radius+1.4cm) {$P^+_{m_1} = P^-_{m_2}$};
\node at ({360/10 *7.5}: \radius+1.4cm) {$P_{\cdot}^-$};
    \node at ({360/10 *8.5}: \radius+1.4cm) {$P_2^-$};
\node at ({360/10 *9.5}: \radius+1.4cm) {$P_1^-$};
    \node[xshift=0.5cm] at ({360/10 *10.5}: \radius+1.4cm) {$P^+_0= P^-_0$};

\end{tikzpicture}
    \caption{Generating objects in the partially wrapped category of the annulus
    with $(m_1,m_2)$ marked points. }
    \label{fig3}
\end{figure}
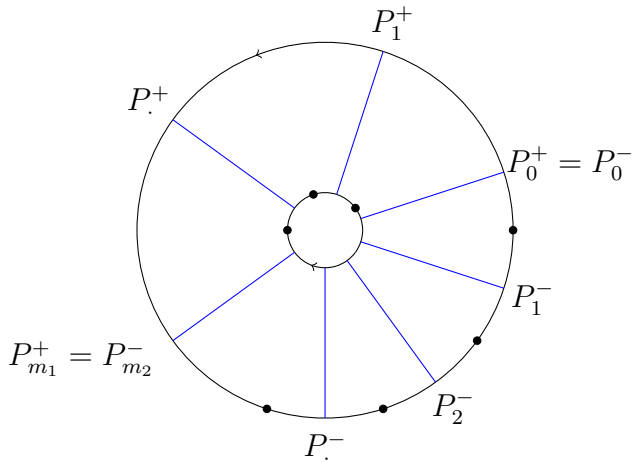

The corresponding endomorphism
algebra between the generators is the path algebra of the quiver drawn below.

\begin{equation}
    \begin{tikzpicture}[baseline=-0.8ex]
\matrix (m) [matrix of math nodes,row sep=3em,column sep=2.8em,minimum width=2em]
    {P^+_0 & P^+_1 &\cdots&P^+_{m_1-1}&P^+_{m_1}\\
    P^-_0 & P^-_1 &\cdots&P^-_{m_2-1}&P^-_{m_2}\\};

 \path[-stealth]
    (m-1-1) edge node [left] {\tiny $=$} (m-2-1)
            edge node [above] {} (m-1-2)
    (m-1-2) edge node [above] {} (m-1-3)
    (m-1-3) edge node [above] {} (m-1-4)
    (m-1-4) edge node [above] {} (m-1-5)

    (m-1-5) edge node [left] {\tiny $=$} (m-2-5)

    (m-2-1) edge node [above] {} (m-2-2)
    (m-2-2) edge node [above] {} (m-2-3)
    (m-2-3) edge node [above] {} (m-2-4)
    (m-2-4) edge node [above] {} (m-2-5);

\end{tikzpicture}
\end{equation}

We will next describe how to glue several copies of $\mathcal{W}(0;m_1,m_2)$ to obtain more interesting computations. We start with the following special case.

\subsection{Computation of the partially wrapped Fukaya category for linear gluing}\label{W0-sec}

We next study the case of a genus $0$ surface where two of the boundary holes are distinguished and
allowed to have arbitrarily many marked points. We denote the number of these marked point by $r_0 $
and $r_n$. The remaining boundary holes have exactly $2$ marked points each. 

As auxiliary data, we choose positive integers $r_1,r_2,\ldots, r_{n-1}$ so that the total number of
holes is 
\[ N= 1+r_1+r_2+\ldots+r_{n-1}+1. \]
We consider the derived category of $\mathcal{W}(0;r_0,(2)^{N-2},r_n)$ which depends only on the numbers
$r_0$, $r_n$ and $N$. However, we use the choice of $r_1,\ldots,r_{n-1}$ in
constructing a strong exceptional collection as in Figure \ref{fig4}. 

\begin{figure}[!htbp]
\centering
    \begin{tikzpicture}[thick,scale=0.8, every node/.style={transform shape}]
\tikzset{vertex/.style = {style=circle,draw, fill,  minimum size = 2pt,inner        sep=1.5pt}}
    \tikzset{->-/.style={decoration={ markings,
        mark=at position #1 with {\arrow{>}}},postaction={decorate}}} 

    \draw [thick=1.5] (1,0) arc (0:90:1);
    \draw [->-=.5, thick=1.5] (1,12) arc (360:270:1);
    \draw [thick=1.5](15,12) arc (180:270:1);
    \draw [->-=.6, thick=1.5](15,0) arc (180:90:1);

\draw  [thick=1.5]((5,0) arc (0:180:0.7);
\draw  [thick=1.5]((9,0) arc (0:180:0.7);
\draw  [thick=1.5]((13,0) arc (0:180:0.7);

\draw  [->-=.4, thick=1.5]((5,12) arc (0:-180:0.7);
\draw  [->-=.4, thick=1.5]((9,12) arc (0:-180:0.7);
\draw  [->-=.4, thick=1.5]((13,12) arc (0:-180:0.7);

\draw  [->-=.65, thick=1.5]((4.8,4) arc (360:0:0.5);
\draw  [->-=.65, thick=1.5]((4.8,8) arc (360:0:0.5);

\draw  [->-=.65, thick=1.5]((8.8,3) arc (360:0:0.5);
\draw  [->-=.65,  thick=1.5]((8.8,6) arc (360:0:0.5);
\draw  [->-=.65, thick=1.5]((8.8,9) arc (360:0:0.5);

\draw  [->-=.65, thick=1.5]((12.8,4) arc (360:0:0.5);
\draw  [->-=.65, thick=1.5]((12.8,8) arc (360:0:0.5);

\draw [thick=1.5] (0,7.5) -- (0,11);
\draw [thick=1.5,dashed] (0,7.5) -- (0,6);
\draw [thick=1.5] (0,1) -- (0,6);

\draw [thick=1.5] (16,4.5) -- (16,11);
\draw [thick=1.5, dashed] (16,4.5) -- (16,3);
\draw [thick=1.5] (16,1) -- (16,3);

\draw [blue] (1,0) -- (3.6,0);
\node at (2.3,0.3) {\tiny $P^+_{1,r_1}$};
\draw [blue] (5,0) -- (7.6,0);
\node at (6.3,0.3) {\tiny $P^+_{2,r_2}$};
 \draw [blue,dashed] (9,0) -- (11.6,0);
\draw [blue] (13,0) -- (15,0);
\node at (14.3,0.3) {\tiny $P^+_{n,r_n}$};
 
\draw [blue] (1,12) -- (3.6,12);
\node at (2.3,11.7) {\tiny $P^-_{1,r_0}$};
 \draw [blue] (5,12) -- (7.6,12);
\draw [blue,dashed] (9,12) -- (11.6,12);
\node at (6.3,11.7) {\tiny $P^-_{2,r_1}$};
\draw [blue] (13,12) -- (15,12);
    \node at (14.3,11.7) {\tiny $P^-_{n,r_{n-1}}$};

\draw[blue] (4.3,8.5) -- (4.3,11.3);
    \node at (4.6, 9.8) {\tiny $S_{1,0}$};
\draw[blue, dashed] (4.3,4.5) -- (4.3,7.5);
\draw[blue] (4.3,0.7) -- (4.3,3.5);
    \node at (3.7, 1.6) {\tiny $S_{1, r_1 -1}$};

\draw[blue] (8.3,9.5) -- (8.3,11.3);
    \node at (8.6, 10) {\tiny $S_{2,0}$};
\draw[blue] (8.3,6.5) -- (8.3,8.5);
    \node at (8.6, 7.5) {\tiny $S_{2,1}$};
\draw[blue, dashed] (8.3,3.5) -- (8.3,5.5);
\draw[blue] (8.3,0.7) -- (8.3,2.5);
\node at (8.9, 1.6) {\tiny $S_{2,r_2 -1}$};

\draw[blue] (12.3,8.5) -- (12.3,11.3);
\node at (11.8, 9.9) {\tiny $S_{n-1,0}$};
\draw[blue, dashed] (12.3,4.5) -- (12.3,7.5);
\draw[blue] (12.3,0.7) -- (12.3,3.5);
\node at (11.4, 1.6) {\tiny $S_{n-1,r_{n-1}-1}$};

\draw[blue] (0.4,0.91) -- (4.1,11.33);
\node at (1.8, 6) {\tiny $P_{1,0}^\pm$};
\draw[blue] (0.6,0.8) -- (4.1,7.53);
\node at (3.2, 5) {\tiny $P_{1,1}^+$};
\draw[blue] (0.8,0.59) --  (4.1,3.53);
\node at (2.35, 2.6) {\tiny $P_{1,r_1-1}^+$};
\draw[blue] (0,4.46) -- (3.8,11.49);
\node at (1.55, 8) {\tiny $P_{1,1}^-$};
\draw[blue] (0,8.46) -- (3.61,11.8);
\node at (1.35, 10.3) {\tiny $P_{1,r_0-1}^-$};

\draw[blue] (4.5,0.66) -- (8.1,11.33);
\node at (5.9, 6) {\tiny $P_{2,0}^{\pm}$};
\draw[blue] (4.6,0.63) -- (8.1,8.52);
\node at (6.9, 5) {\tiny $P_{2,1}^{+}$};
\draw[blue] (4.7,0.59) --  (8.1,5.52);
\node at (6.8, 3) {\tiny $P_{2,2}^{+}$};
\draw[blue] (4.8,0.51) -- (8.1,2.52);
\node at (6.8,1.2) {\tiny $P_{2,{r_2 -1}}^{+}$};

\draw[blue] (4.5,4.46) -- (7.8,11.49);
\node at (5.9,8.2) {\tiny $P_{2,1}^{-}$};

\draw[blue] (4.5,8.46) -- (7.61,11.8);
\node at (5.9,10.7) {\tiny $P_{2,{r_1 -1}}^{-}$};

\draw[blue] (12.5,0.66) -- (15.8,11.01);
\node at (13.8, 6) {\tiny $P_{n,0}^{\pm}$};

\draw[blue] (12.6,0.63) -- (16,8.52);
\node at (15.1, 5.5) {\tiny $P_{n,1}^{+}$};
\draw[blue] (12.7,0.59) --  (16,5.52);
\node at (15.5, 4) {\tiny $P_{n,2}^{+}$};
\draw[blue] (12.8,0.51) -- (16,2.52);
\node at (14.8, 2.3) {\tiny $P_{n,{r_n-1}}^{+}$};

\draw[blue] (12.5,4.46) -- (15.5,11.15);
\node at (13.7,8) {\tiny $P_{n,1}^{-}$};
\draw[blue] (12.5,8.46) -- (15.15,11.5);
\node at (13.9,11) {\tiny $P_{n,r_{n-1}-1}^{-}$};

\node[vertex] at  (0,2.73) {};
\node[vertex] at  (0,8) {};
\node[vertex] at  (0,9.73) {};

\node[vertex] at (16,9.76) {};
\node[vertex] at (16,7.02) {};
\node[vertex] at (16,4.9) {};

\node[vertex] at (4.8,11.49) {};
\node[vertex] at (8.8,11.49) {};
\node[vertex] at (12.8,11.49) {} ;

\node[vertex] at (3.8,0.51) {};
\node[vertex] at (7.8,0.51) {};
\node[vertex] at (11.8,0.51) {} ;
\node[vertex] at (15.15,0.51) {} ;

\node[vertex] at (3.8,4) {};
\node[vertex] at (4.8,4) {};
\node[vertex] at (3.8,8) {};
\node[vertex] at (4.8,8) {};

\node[vertex] at (7.8,3) {};
\node[vertex] at (8.8,3) {};
\node[vertex] at (7.8,6) {};
\node[vertex] at (8.8,6) {};
\node[vertex] at (7.8,9) {};
\node[vertex] at (8.8,9) {};

\node[vertex] at (11.8,4) {};
\node[vertex] at (12.8,4) {};
\node[vertex] at (11.8,8) {};
\node[vertex] at (12.8,8) {};

\end{tikzpicture}
    \caption{Generating objects in $\mathcal{W}(0;r_0,2,2,\ldots,2,r_n)$. Top and bottom are
    identified. }
\label{fig4}
\end{figure}
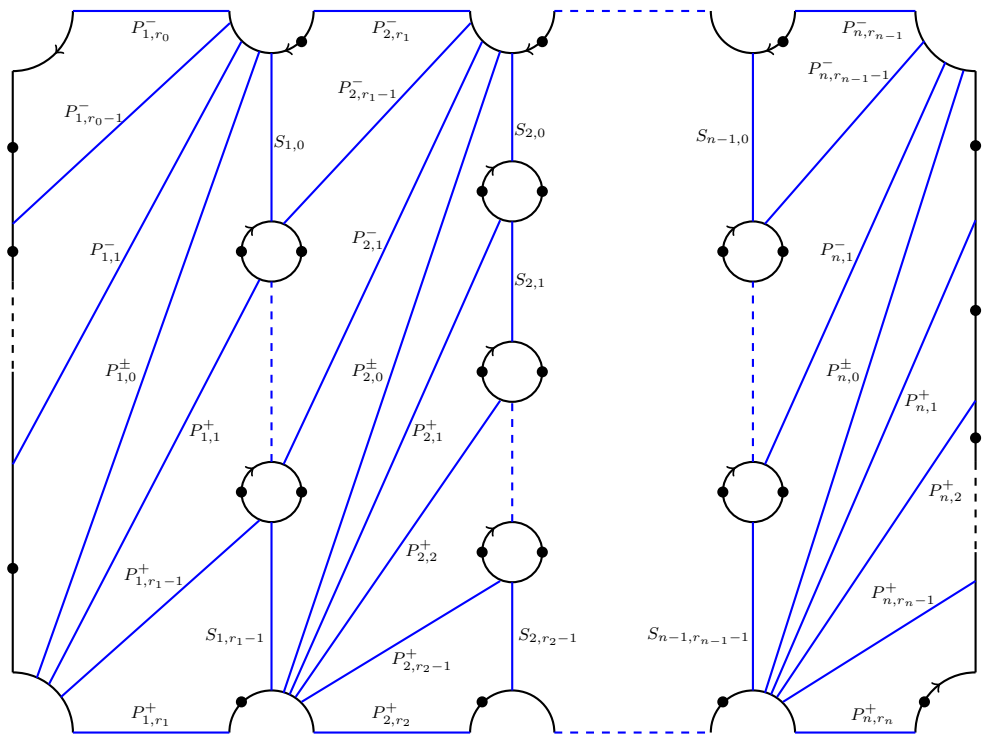

It is easy to observe from Figure \ref{fig4} that the complement of the Lagrangians drawn consist of
disks with precisely one marked point at each boundary. Hence, the objects drawn generate the
partially wrapped Fukaya category $\mathcal{W}(0;r_0,2,2,\ldots,2,r_n)$. The corresponding
quiver algebra is given in Figure \ref{fig5}. The only relations are given by the quadratic relations
\[ ya =0 \text{\ and \ } xb=0 \]
whenever the composition is possible. 

\begin{figure}[!htbp]
\begin{equation}
\begin{tikzpicture}[baseline=-0.8ex]
\matrix (m) [matrix of math nodes,row sep=1.6em,column sep=2em,minimum width=2em]
    {&P^-_{1,0} & P^-_{1,1}&\cdots&P^-_{1,r_0-1}&P^-_{1,r_0}\\
     &P^+_{1,0} & P^+_{1,1}&\cdots&P^+_{1,r_1-1}&P^+_{1,r_1}\\
     &S_{1,0}   & S_{1,1}  &\cdots&S_{1,r_1-1}  &           \\
    P^-_{2,r_1}&P^-_{2,r_1-1}&P^-_{2,r_1-2}&\cdots&P^-_{2,0}& \\
    P^+_{2,r_2}&P^+_{2,r_2-1}&P^+_{2,r_2-2}&\cdots&P^+_{2,0}& \\
    &S_{2,r_2 -1} &S_{2,r_2 -2} & \cdots & S_{2,0}& \\
     & \cdots & \cdots & \cdots   & \cdots & \\
     & \cdots & \cdots & \cdots   & \cdots & \\
    &S_{n-1,r_{n-1}-1} & S_{n-1,r_{n-1}-2} & \cdots &S_{n-1,0}& \\
    & P^-_{n,0} & P^-_{n,1} &\cdots&P^-_{n,r_{n-1}-1}&P^-_{n, r_{n-1}}\\
    & P^+_{n,0} & P^+_{n,1} &\cdots&P^+_{n,r_{n}-1}&P^+_{n,r_{n}} \\};
        
     \path[-stealth]
       (m-1-2) edge node [left] {\scriptsize $=$} (m-2-2)
               edge node [above] {\scriptsize $x_{1,0}$} (m-1-3)
       (m-1-3) edge node [above] {\scriptsize $x_{1,1}$} (m-1-4)
       (m-1-4) edge node [above] {} (m-1-5)
       (m-1-5) edge node [above] {\scriptsize $x_{1,r_0-1}$} (m-1-6)
       
       (m-1-6) edge node [left] {\scriptsize $=$} (m-2-6)

       (m-2-2) edge node [above] {\scriptsize $y_{1,0}$} (m-2-3)
       (m-2-3) edge node [above] {\scriptsize $y_{1,1}$} (m-2-4)
       (m-2-4) edge node [above] {} (m-2-5)
       (m-2-5) edge node [above] {\scriptsize $y_{1,r_1-1}$} (m-2-6)

       (m-3-2) edge node [left] {\scriptsize $a_{1,0}$} (m-2-2)
       (m-3-3) edge node [left] {\scriptsize $a_{1,1}$} (m-2-3)
       (m-3-5) edge node [left] {\scriptsize $a_{1,r_1-1}$} (m-2-5)

       (m-3-2) edge node [left] {\scriptsize $b_{1,0}$} (m-4-2)
       (m-3-3) edge node [left] {\scriptsize $b_{1,1}$} (m-4-3)
       (m-3-5) edge node [left] {\scriptsize $b_{1,r_1-1}$} (m-4-5)

       (m-4-1) edge node [left] {\scriptsize $=$} (m-5-1)
           
       (m-4-2) edge node [above] {\scriptsize $x_{2,r_1-1}$} (m-4-1)
       (m-4-3) edge node [above] {\scriptsize $x_{2,r_1-2}$} (m-4-2)
       (m-4-4) edge node [above] {} (m-4-3)
       (m-4-5) edge node [above] {\scriptsize $x_{2,0}$} (m-4-4)
       
       (m-4-5) edge node [left] {\scriptsize $=$} (m-5-5)

       (m-5-2) edge node [above] {\scriptsize $y_{2,r_2-1}$} (m-5-1)
       (m-5-3) edge node [above] {\scriptsize $y_{2,r_2-2}$} (m-5-2)
       (m-5-4) edge node [above] {} (m-5-3)
       (m-5-5) edge node [above] {\scriptsize $y_{2,0}$} (m-5-4)

       (m-6-2) edge node [left] {\scriptsize $a_{2,r_2-1}$} (m-5-2)
       (m-6-3) edge node [left] {\scriptsize $a_{2,r_2-2}$} (m-5-3)
       (m-6-5) edge node [left] {\scriptsize $a_{2,0}$} (m-5-5)

       (m-6-2) edge node [left] {\scriptsize $b_{2,r_2-1}$} (m-7-2)
       (m-6-3) edge node [left] {\scriptsize $b_{2,r_2-2}$} (m-7-3)
       (m-6-5) edge node [left] {\scriptsize $b_{2,0}$} (m-7-5)

       (m-9-2) edge node [left] {\scriptsize $a_{n-1,r_{n-1}-1}$} (m-8-2)
       (m-9-3) edge node [left] {\scriptsize $a_{n-1,r_{n-1}-2}$} (m-8-3)
       (m-9-5) edge node [left] {\scriptsize $a_{n-1,0}$} (m-8-5)

       (m-9-2) edge node [left] {\scriptsize $b_{n-1,r_{n-1}-1}$} (m-10-2)
       (m-9-3) edge node [left] {\scriptsize $b_{n-1,r_{n-1}-2}$} (m-10-3)
       (m-9-5) edge node [left] {\scriptsize $b_{n-1,0}$} (m-10-5)

       (m-10-2) edge node [left] {\scriptsize $=$} (m-11-2)
              edge node [above] {\scriptsize $x_{n,0}$} (m-10-3)
       (m-10-3) edge node [above] {\scriptsize $x_{n,1}$} (m-10-4)
       (m-10-4) edge node [above] {} (m-10-5)
       (m-10-5) edge node [above] {\scriptsize $x_{n,r_n-1}$} (m-10-6)
       
       (m-10-6) edge node [left] {\scriptsize $=$} (m-11-6)

       (m-11-2) edge node [above] {\scriptsize $y_{n,0}$} (m-11-3)
       (m-11-3) edge node [above] {\scriptsize $y_{n,1}$} (m-11-4)
       (m-11-4) edge node [above] {} (m-11-5)
       (m-11-5) edge node [above] {\scriptsize $y_{n,r_n-1}$} (m-11-6);

\end{tikzpicture}
\end{equation}
    \caption{Quiver describing $\mathcal{W}(0;r_0,2,2,\ldots,2,r_n)$} 
    \label{fig5}
\end{figure}

Next, we are going to modify our surface along with the exceptional collection.
One can note from Figure \ref{fig5} that there are full and faithful embeddings 
\[  \mathcal{W}(0;r_i,r_{i+1}) \to \mathcal{W}(0;r_0,2,2,\ldots,2, r_n) \] 
for $i=0,1,\ldots, n-1$.
Indeed, the genus 0 surface in Figure \ref{fig4} is constructed by connecting $n$ annuli along
strips which are given by tubular neighborhoods of curves $S_{i,j}$. Now, in attaching these
strips a choice is made: the strips are attached in the most obvious way as in the left part of
Figure \ref{fignew}. In general, a more complicated attachment of these strips
are encoded by a sequence of permutations $(\sigma_1,\sigma_2,\ldots, \sigma_{n-1}) \in
\mathfrak{S}_{r_1} \times \mathfrak{S}_{r_2} \times \ldots, \times
\mathfrak{S}_{r_{n-1}}$ where
$\mathfrak{S}_{r_i}$ is the permutation group on $r_i$ elements. The effect of a transposition
on the construction of the surface is described in Figure \ref{fignew}. In general, this will change the
topological type of the surface. An example is given in Figure \ref{fignew2}. We omit the proof
of the following elementary proposition which determines the topological type of the resulting
surface and the distribution of the marked points in terms of the data of the permutations
used in attaching the strips. 

\begin{prop} \label{topology} Suppose that the attachments of strips are made using the set of permutations
    $\sigma=(\sigma_1,\sigma_2,\ldots,\sigma_{n-1}) \in \mathfrak{S}_{r_1} \times
    \mathfrak{S}_{r_2} \times \ldots \times \mathfrak{S}_{r_{n-1}}$, and let $\tau = (\tau_1,\tau_2,\ldots,
    \tau_{n-1}) \in \mathfrak{S}_{r_1} \times
    \mathfrak{S}_{r_2} \times \ldots \times \mathfrak{S}_{r_{n-1}}$, be the set of permutations given by $\tau_i (j) = j-1$ for all $j \in
    \mathbb{Z}/r_i$. 
    The number of boundary components of the resulting surface $X$ is equal to 
    \[ d = 2+ \sum_{i=1}^{n-1} \sum_{k=1}^{r_i} d_{ik} \]
    where $d_{ik}$ is the number of $k$-cycles in the cycle decomposition of $[\sigma_i,\tau_i]$. 
    We have two special boundary components, equipped with $r_0$ and $r_n$ components respectively.
    The remaining components are in bijection with the cycles in cycle decompositions of 
    $[\sigma_i, \tau_i]$ for $i=1, \ldots, n-1$. A component corresponding to a $k$-cycle is equipped with $2k$ marked points.

Finally, the genus $g$ of $X$ can be computed using the following formula for the Euler characteristic of $X$ given by:
    \[ \chi(X) = 2-2g - d = - \sum_{i=1}^{n-1} r_i. \] 
\qed
\end{prop}

Note that changing the permutations does not affect the Euler characteristic of the underlying
topological surface since different permutations are related by cutting and gluing the strips. Note
also that by \cite[Thm. 5.1]{HKK}, the Grothendieck group $K_0(\WW(X, \Lambda))$ is
isomorphic to $H_1(X, \partial X \setminus \Lambda)$ and the rank of the latter group is given by 
\[ \# |\Lambda| - \chi(X), \] 
when $\Lambda \neq \emptyset$. Using Prop. \ref{topology}, this number can
be computed in the above case as:
\[ r_0 + \sum_{i=1}^{n-1} \sum_{k=1}^{r_i} 2k d_{ik} + r_n + \sum_{i=1}^{n-1} r_i =  r_0 + 3
\sum_{i=1}^{n-1} r_i + r_n \]
which is equal to the number of objects given in Figure \ref{fig5} as it should.

The resulting algebra of our generators has a quiver description that is very similar to Figure
\ref{fig5}. The only modification needed is in the target of the maps $b_{i,j}$. Namely, if we modify the
attaching strips according to a permutation $(\sigma_1,\sigma_2,\ldots,
\sigma_{n-1})$, then in Figure \ref{fig5}, we need to let

\[ b_{i,j} : S_{i,j} \to P^-_{i+1, r_i -1 - \sigma_{i}(j)} \]

\begin{figure}[!htbp]
\centering
    \begin{tikzpicture}[thick,scale=0.8, every node/.style={transform shape}]
\tikzset{vertex/.style = {style=circle,draw, fill,  minimum size = 2pt,inner        sep=1pt}}
    \tikzset{->-/.style={decoration={ markings,
        mark=at position #1 with {\arrow{>}}},postaction={decorate}}} 
\tikzset{->-/.style={decoration={ markings,
        mark=at position #1 with {\arrow{>}}},postaction={decorate}}} 
\def \radius {1.5cm}

\draw (1.45,-5.278) to (4.55,-5.278);

\draw [rounded corners=1mm] (1.45,-3.654)--(1.45,-4.466)-- (4.55,-4.466) -- (4.55,-3.654) --cycle;
\draw [rounded corners=1mm] (1.45,6.09)--(1.45,5.278)-- (4.55,5.278) -- (4.55,6.09) --cycle;
\draw [rounded corners=1mm] (1.45,2.842)--(1.45,2.03)-- (4.55,2.03) -- (4.55,2.842) --cycle;
\draw [rounded corners=1mm] (1.45,0.406)--(1.45,1.218)-- (4.55,1.218) -- (4.55,0.406) --cycle;
\draw [rounded corners=1mm] (1.45,-0.406)--(1.45,-1.218)-- (4.55,-1.218) -- (4.55,-0.406) --cycle;

\draw (1.45,6.902) to (4.55,6.902);

\draw (1.45,2.842) to (4.55,2.842);
\draw (1.45,3.654) to (4.55,3.654);

\draw (1.45,-1.218) to (4.55,-1.218);
\draw (1.45,-2.03) to (4.55,-2.03);

\draw[blue] (3,-0.406) to (3,0.406);
\draw[blue] (3,1.218) to (3,2.03);
\draw[blue] (3,-1.218) to (3,-2.03);
\draw[blue] (3,2.842) to (3,3.654);
\draw[blue] (3,-4.466) to (3,-5.278);
\draw[blue] (3,6.902) to (3,6.09);

\draw[dashed] (3,5.1) to (3,3.8);
\draw[dashed] (3,-3.5) to (3,-2);

\node at (3.4, 6.5) {\tiny $S_{i,0}$};
\node at (3.5, 3.2) {\tiny $S_{i,j-1}$};
\node at (3.4, 1.6) {\tiny $S_{i,j}$};
\node at (3.5, 0) {\tiny $S_{i,j+1}$};
\node at (3.5, -1.6) {\tiny $S_{i,j+2}$};
\node at (3.6, -4.8) {\tiny $S_{i,r_i-1}$};

\foreach \s in {0,2,4,6,8,10,14} {
    \node[vertex][yshift= {(\s)*0.812cm}] at (2, -5.278){};
}

\foreach \s in {0,4,6,8,10,12, 14} {
    \node[vertex][yshift= {(\s)*0.812cm}] at (4, -4.466){};}

\draw[->-=1, thick=1.5](7,0.812) -- (9,0.812);

    \begin{scope}[xshift=10cm]

\draw (1.45,-5.278) to (4.55,-5.278);

\draw [rounded corners=1mm] (1.45,-3.654)--(1.45,-4.466)-- (4.55,-4.466) -- (4.55,-3.654) --cycle;
\draw [rounded corners=1mm] (1.45,6.09)--(1.45,5.278)-- (4.55,5.278) -- (4.55,6.09) --cycle;

\draw (1.45,6.902) to (4.55,6.902);

\draw (1.45,2.842) to (4.55,2.842);
\draw (1.45,3.654) to (4.55,3.654);

\draw (1.45,2.03) to[in=180,out=0] (4.55,0.406);
\draw (1.45,1.218) to[in=180,out=0] (4.55,-0.406);

\draw (1.45,0.406) to [in=210,out=0] (2.35,0.83);
\draw (2.95,1.33) to [in=180,out=30] (4.55,2.03);

\draw (1.45,-0.406) to [in=210,out=0] (2.85,0.3);
\draw (3.6,0.78) to [in=180,out=30] (4.55,1.218);

\draw (1.45,-1.218) to (4.55,-1.218);
\draw (1.45,-2.03) to (4.55,-2.03);

\draw(1.45,1.218) to (1.45,0.406);
 \draw(4.55,1.218) to (4.55,0.406);
 \draw(1.45,-1.218) to (1.45,-0.406);
 \draw(4.55,-1.218) to (4.55,-0.406);
 \draw(1.45,2.03) to (1.45,2.842);
 \draw(4.55,2.03) to (4.55,2.842);

\draw[blue] (3.6,0.77) to (3.6,1.72);
\draw[blue] (2.5,0.83) to (2.5,1.642);
\draw[blue] (3,-1.218) to (3,-2.03);
\draw[blue] (3,2.842) to (3,3.654);
\draw[blue] (3,-4.466) to (3,-5.278);
\draw[blue] (3,6.902) to (3,6.09);

\draw[dashed] (3,5.1) to (3,3.8);

\draw[dashed] (3,-3.5) to (3,-2);

\node at (3.4, 6.5) {\tiny $S_{i,0}$};
\node at (3.5, 3.2) {\tiny $S_{i,j-1}$};
\node at (2.1, 1.45) {\tiny $S_{i,j}$};
\node at (4.1, 1.4) {\tiny $S_{i,j+1}$};
\node at (3.5, -1.6) {\tiny $S_{i,j+2}$};
\node at (3.6, -4.8) {\tiny $S_{i,r_i-1}$};

\foreach \s in {0,2,4,10,14} {
    \node[vertex][yshift= {(\s)*0.812cm}] at (2, -5.278){};
}

 \node[vertex][yshift= {(6)*0.812cm}] at (2, -5.15){};
 \node[vertex][yshift= {(8)*0.812cm}] at (2, -5.378){};

\foreach \s in {0,4,10,12, 14} {
    \node[vertex][yshift= {(\s)*0.812cm}] at (4, -4.466){};}

  \node[vertex][yshift= {(6)*0.812cm}] at (4, -4.366){};
 \node[vertex][yshift= {(8)*0.812cm}] at (4, -4.566){};

\end{scope}

    \end{tikzpicture}
\caption{Effect of the permutation $(j,j+1)$ in $\mathfrak{S}_{r_i}$} 
\label{fignew}

\end{figure}
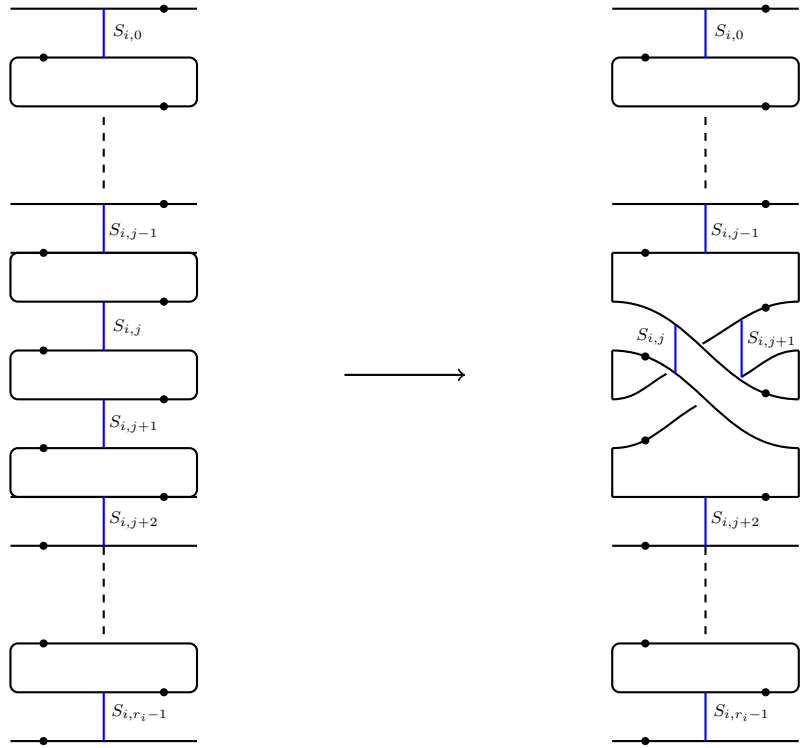 

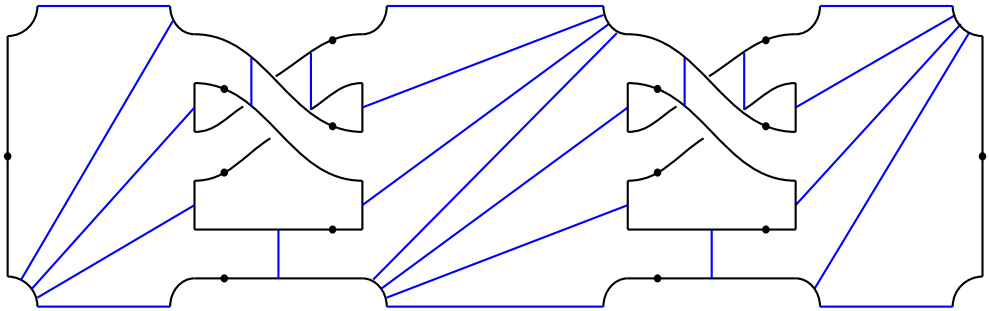
\begin{figure}[!htbp]
\centering
    \begin{tikzpicture}[thick,scale=0.8, every node/.style={transform shape}]
\tikzset{vertex/.style = {style=circle,draw, fill,  minimum size = 2pt,inner        sep=1pt}}
    \tikzset{->-/.style={decoration={ markings,
        mark=at position #1 with {\arrow{>}}},postaction={decorate}}} 
\tikzset{->-/.style={decoration={ markings,
        mark=at position #1 with {\arrow{>}}},postaction={decorate}}} 
\def \radius {1.5cm}

\begin{scope}[xscale=0.9, yscale=1]
 




\draw (1.45,2.03) to [in=270,out=180] (1,2.5); 
\draw (4.55,2.03) to [in=270,out=0] (5,2.5); 

\draw (1.45,-2.03) to [in=90,out=180] (1,-2.5); 
\draw (4.55,-2.03) to [in=90,out=0] (5,-2.5);

\draw (1.45,2.03) to[in=180,out=0] (4.55,0.406);
\draw (1.45,1.218) to[in=180,out=0] (4.55,-0.406);

\draw (1.45,0.406) to [in=210,out=0] (2.35,0.83);
\draw (2.95,1.33) to [in=180,out=30] (4.55,2.03);

\draw (1.45,-0.406) to [in=210,out=0] (2.85,0.3);
\draw (3.6,0.78) to [in=180,out=30] (4.55,1.218);

    \draw (1.45,-1.218) to (4.55,-1.218);
\draw (1.45,-2.03) to (4.55,-2.03);

\draw(1.45,1.218) to (1.45,0.406);
 \draw(4.55,1.218) to (4.55,0.406);
 \draw(1.45,-1.218) to (1.45,-0.406);
 \draw(4.55,-1.218) to (4.55,-0.406);

\draw[blue] (3.6,0.77) to (3.6,1.72);
\draw[blue] (2.5,0.83) to (2.5,1.642);
\draw[blue] (3,-1.218) to (3,-2.03);

\draw (-2, -2) to (-2, 2); 
\draw (-2,2) to [in=270,out=0] (-1.45,2.5); 
\draw (-2,-2) to [in=90,out=0] (-1.45,-2.5); 

\draw[blue] (-1.45,2.5) to (1,2.5);
\draw[blue] (-1.45,-2.5) to (1,-2.5);

\draw[blue] (-1.75, -2.05) to (1.05, 2.25);
\draw[blue] (-1.55, -2.2) to (1.45, 0.812);
\draw[blue] (-1.45, -2.35) to (1.45, -0.812);

\draw[blue](-1.95, -0.812) +(6.5,0) to (9.1, 2.2);
\draw[blue](-1.95, 0.812) +(6.5,0) to (9, 2.35);

\draw[blue](-1.75, -2.05) +(6.5,0) to (9.25, 2.05);
\draw[blue](-1.55, -2.2)+(6.45,0) to (9.45, 0.812);
\draw[blue](-1.45, -2.35)+(6.45,0) to (9.45, -0.812);

\draw[blue] (-1.55, -2.2)+(14.45,0) to (15.75, 2.05);
\draw[blue] (12.55,-0.812) to (15.6, 2.2);
\draw[blue] (12.55,0.812) to (15.5, 2.35);

\node[vertex] at (2, 1.12){}; 
\node[vertex] at (2, -0.27){}; 
\node[vertex] at (2, -2.03){}; 

\node[vertex] at (4, 1.93){}; 
\node[vertex] at (4, 0.5){};

\node[vertex] at (4, -1.218){}; 

\node[vertex] at (-2 , 0){};

\begin{scope}[xshift=8cm]

\draw (1.45,2.03) to [in=270,out=180] (1,2.5); 
\draw (4.55,2.03) to [in=270,out=0] (5,2.5); 

\draw (1.45,-2.03) to [in=90,out=180] (1,-2.5); 
\draw (4.55,-2.03) to [in=90,out=0] (5,-2.5);

\draw (1.45,2.03) to[in=180,out=0] (4.55,0.406);
\draw (1.45,1.218) to[in=180,out=0] (4.55,-0.406);

\draw (1.45,0.406) to [in=210,out=0] (2.35,0.83);
\draw (2.95,1.33) to [in=180,out=30] (4.55,2.03);

\draw (1.45,-0.406) to [in=210,out=0] (2.85,0.3);
\draw (3.6,0.78) to [in=180,out=30] (4.55,1.218);

    \draw (1.45,-1.218) to (4.55,-1.218);
\draw (1.45,-2.03) to (4.55,-2.03);

\draw(1.45,1.218) to (1.45,0.406);
 \draw(4.55,1.218) to (4.55,0.406);
 \draw(1.45,-1.218) to (1.45,-0.406);
 \draw(4.55,-1.218) to (4.55,-0.406);

\draw[blue] (3.6,0.77) to (3.6,1.72);
\draw[blue] (2.5,0.83) to (2.5,1.642);
\draw[blue] (3,-1.218) to (3,-2.03);

\draw (8, -2) to (8, 2); 
\draw (8,2) to [in=270,out=180] (7.45,2.5); 
\draw (8,-2) to [in=90,out=180] (7.45,-2.5); 

\draw[blue] (7.45,2.5) to (5,2.5);
\draw[blue] (7.45,-2.5) to (5,-2.5);

\node[vertex] at (2, 1.12){}; 
\node[vertex] at (2, -0.27){}; 
\node[vertex] at (2, -2.03){}; 

\node[vertex] at (4, 1.93){}; 
\node[vertex] at (4, 0.5){};

\node[vertex] at (4, -1.218){}; 

\node[vertex] at (8 , 0){};

 \end{scope} 

\draw[blue] (9,2.5) to (5,2.5);
\draw[blue] (9,-2.5) to (5,-2.5);

\end{scope}

\end{tikzpicture}

\caption{Connecting three annuli according to the permutations
$\sigma_{1}=\sigma_2 : (1,2,3) \to (2,1,3)$. Generating objects in $\WW(2;1,6,6,1)$} 

\label{fignew2}

\end{figure}

\subsection{Computation of the partially wrapped Fukaya category for circular gluing}\label{W1-sec}

We start with the case of a punctured torus and then will consider a modification leading to higher genus surfaces. 

In Figure \ref{fig8} we depicted the $n$-punctured torus with $2$-marked points at each boundary
components. As before, we choose auxiliary data given by integers $r_0,r_1,\ldots, r_{n-1}$ and we
also write $r_n=r_0$. The derived category of $\mathcal{W}(1;(2)^N)$ only depends on the
total number of holes 
\[ N = r_0 + r_1+ \ldots + r_{n-1}. \]
Note that each boundary hole has exactly 2 marked points.  

Again, it is easy to observe from Figure \ref{fig8} that the complement of the Lagrangians drawn consists of
disks with precisely one marked point at each boundary. Hence, the objects drawn generate the
partially wrapped Fukaya category $\mathcal{W}(1;2,2,\ldots,2)$. The corresponding
quiver algebra is given in Figure \ref{fig9}. The only relations are given by the quadratic relations
\[ ya =0 \text{\ and \ } xb=0 \]
whenever the composition is possible.

As in Section \ref{W0-sec}, we can do a more complicated attachment of bands that form the tubular
neighborhood of the objects $S_{i,j}$ using a set of permutations $\sigma = (\sigma_0,
\sigma_1,\ldots, \sigma_{n-1}) \in \mathfrak{S}_{r_0} \times \mathfrak{S}_{r_1} \times \ldots
\mathfrak{S}_{r_{n-1}}$. The topology of the resulting surface is determined by the following
analogue of Prop. \ref{topology}.

\begin{prop} \label{topology1} Suppose that the attachments of strips are made using the set of permutations
    $\sigma=(\sigma_0,\sigma_1,\ldots,\sigma_{n-1}) \in \mathfrak{S}_{r_0} \times
    \mathfrak{S}_{r_1} \times \ldots \times \mathfrak{S}_{r_{n-1}}$, and let $\tau =
    (\tau_0,\tau_1,\ldots, \tau_{n-1}) \in \mathfrak{S}_{r_0} \times
    \mathfrak{S}_{r_1} \times \ldots \times \mathfrak{S}_{r_{n-1}}$, be the set of permutations given by $\tau_i (j) = j-1$ for all $j \in
    \mathbb{Z}/r_i$. 
    The number of boundary components of the resulting surface $X$ is equal to 
    \[ d = \sum_{i=1}^{n-1} \sum_{k=1}^{r_i} d_{ik} \]
    where $d_{ik}$ is the number of $k$-cycles in the cycle decomposition of $[\sigma_i,\tau_i]$. 
        The boundary components are in bijection with the cycles in cycle decompositions of 
    $[\sigma_i, \tau_i]$ for $i=0, \ldots, n-1$, where a component corresponding to a $k$-cycle is equipped with $2k$ marked points.

Finally, the genus $g$ of $X$ can be computed using the following formula for the Euler characteristic of $X$ given by:
    \[ \chi(X) = 2-2g - d = - \sum_{i=0}^{n-1} r_i. \] 
\qed
\end{prop}

Again by \cite[Thm. 5.1]{HKK}, the rank of the Grothendieck group $K_0(\WW(X, \Lambda))$ can
be computed in the above case as:
\[ \sum_{i=1}^{n-1} \sum_{k=0}^{r_i} 2k d_{ik} + \sum_{i=0}^{n-1} r_i =  3
\sum_{i=0}^{n-1} r_i \]
which is equal to the number of objects given in Figure \ref{fig5} as it should.

Finally, as in the previous section, the resulting algebra of our generators has a quiver description that is very similar to Figure
\ref{fig9}. The only modification needed is in the target of the maps $b_{i,j}$. Namely, if we modify the
attaching strips according to a permutation $(\sigma_0,\sigma_1,\ldots,
\sigma_{n-1})$, then in Figure \ref{fig9}, we need to let

\[ b_{i,j} : S_{i,j} \to P^-_{i+1, r_i -1 - \sigma_{i}(j)} \]

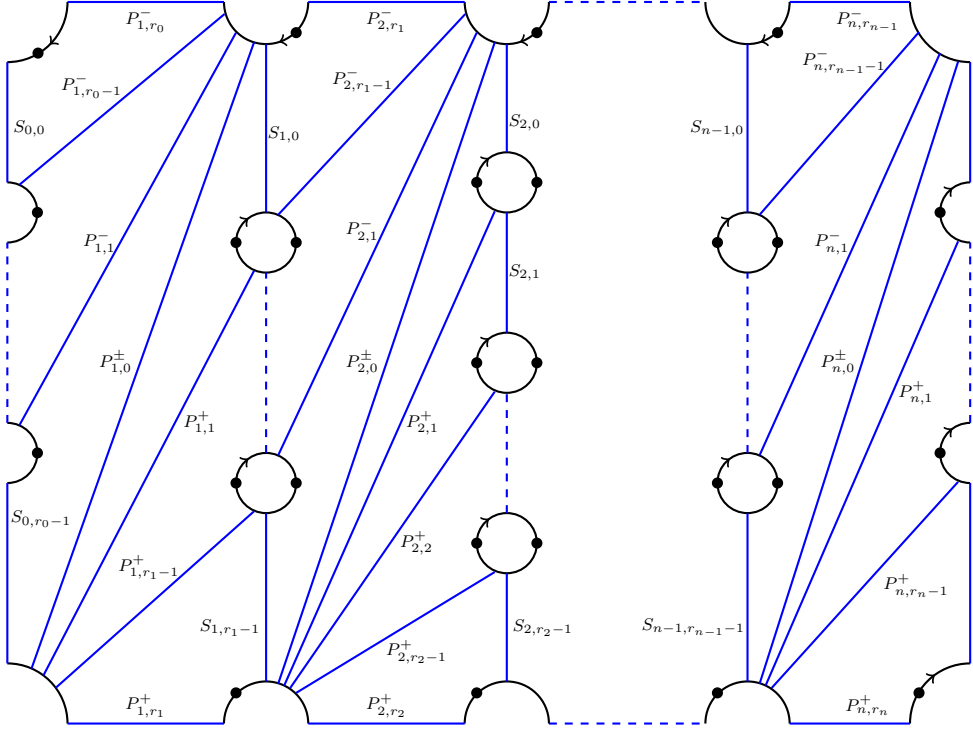
\begin{figure}[!htbp]
\centering
\begin{tikzpicture}[thick,scale=0.8, every node/.style={transform shape}]

\tikzset{vertex/.style = {style=circle,draw, fill,  minimum size = 2pt,inner        sep=1.5pt}}
    \tikzset{->-/.style={decoration={ markings,
        mark=at position #1 with {\arrow{>}}},postaction={decorate}}} 

    \draw [thick=1.5] (1,0) arc (0:90:1);
    \draw [->-=.5, thick=1.5] (1,12) arc (360:270:1);
    \draw [thick=1.5](15,12) arc (180:270:1);
    \draw [->-=.6, thick=1.5](15,0) arc (180:90:1);

\draw  [thick=1.5]((5,0) arc (0:180:0.7);
\draw  [thick=1.5]((9,0) arc (0:180:0.7);
\draw  [thick=1.5]((13,0) arc (0:180:0.7);

\draw  [->-=.4, thick=1.5]((5,12) arc (0:-180:0.7);
\draw  [->-=.4, thick=1.5]((9,12) arc (0:-180:0.7);
\draw  [->-=.4, thick=1.5]((13,12) arc (0:-180:0.7);

\draw  [->-=.65, thick=1.5]((4.8,4) arc (360:0:0.5);
\draw  [->-=.65, thick=1.5]((4.8,8) arc (360:0:0.5);

\draw  [->-=.65, thick=1.5]((8.8,3) arc (360:0:0.5);
\draw  [->-=.65,  thick=1.5]((8.8,6) arc (360:0:0.5);
\draw  [->-=.65, thick=1.5]((8.8,9) arc (360:0:0.5);

\draw  [->-=.65, thick=1.5]((12.8,4) arc (360:0:0.5);
\draw  [->-=.65, thick=1.5]((12.8,8) arc (360:0:0.5);

\draw  [thick=1.5]((0,9) arc (90:-90:0.5);
\draw  [thick=1.5]((0,5) arc (90:-90:0.5);

\draw  [->-=.8, thick=1.5]((16,8) arc (270:90:0.5);
\draw  [->-=.8, thick=1.5]((16,4) arc (270:90:0.5);

\node[vertex] at (0.51,11.15) {} ;
\node[vertex] at  (0.5,8.5) {};
\node[vertex] at  (0.5,4.5) {};

\node[vertex] at  (15.5,8.5) {};
\node[vertex] at  (15.5,4.5) {};

\draw [blue] (1,0) -- (3.6,0);
\node at (2.3,0.3) {\tiny $P^+_{1,r_1}$};
\draw [blue] (5,0) -- (7.6,0);
\node at (6.3,0.3) {\tiny $P^+_{2,r_2}$};
 \draw [blue,dashed] (9,0) -- (11.6,0);
\draw [blue] (13,0) -- (15,0);
\node at (14.3,0.3) {\tiny $P^+_{n,r_n}$};
 
\draw [blue] (1,12) -- (3.6,12);
\node at (2.3,11.7) {\tiny $P^-_{1,r_0}$};
 \draw [blue] (5,12) -- (7.6,12);
\draw [blue,dashed] (9,12) -- (11.6,12);
\node at (6.3,11.7) {\tiny $P^-_{2,r_1}$};
\draw [blue] (13,12) -- (15,12);
    \node at (14.3,11.7) {\tiny $P^-_{n,r_{n-1}}$};

\draw[blue] (4.3,8.5) -- (4.3,11.3);
    \node at (4.6, 9.8) {\tiny $S_{1,0}$};
\draw[blue, dashed] (4.3,4.5) -- (4.3,7.5);
\draw[blue] (4.3,0.7) -- (4.3,3.5);
    \node at (3.7, 1.6) {\tiny $S_{1, r_1 -1}$};

\draw[blue] (8.3,9.5) -- (8.3,11.3);
    \node at (8.6, 10) {\tiny $S_{2,0}$};
\draw[blue] (8.3,6.5) -- (8.3,8.5);
    \node at (8.6, 7.5) {\tiny $S_{2,1}$};
\draw[blue, dashed] (8.3,3.5) -- (8.3,5.5);
\draw[blue] (8.3,0.7) -- (8.3,2.5);
\node at (8.9, 1.6) {\tiny $S_{2,r_2 -1}$};

\draw[blue] (12.3,8.5) -- (12.3,11.3);
\node at (11.8, 9.9) {\tiny $S_{n-1,0}$};
\draw[blue, dashed] (12.3,4.5) -- (12.3,7.5);
\draw[blue] (12.3,0.7) -- (12.3,3.5);
\node at (11.4, 1.6) {\tiny $S_{n-1,r_{n-1}-1}$};

\draw[blue] (0,9) -- (0,11);
\node at (0.37, 9.9) {\tiny $S_{0,0}$};
\draw[blue, dashed] (0,5) -- (0,8);
\draw[blue] (0,1) -- (0,4);
\node at (0.56, 3.4) {\tiny $S_{0,r_{0}-1}$};

\draw[blue] (16,9) -- (16,11);
\draw[blue, dashed] (16,5) -- (16,8);
\draw[blue] (16,1) -- (16,4);

\draw[blue] (0.4,0.91) -- (4.1,11.33);
\node at (1.8, 6) {\tiny $P_{1,0}^\pm$};
\draw[blue] (0.6,0.8) -- (4.1,7.53);
\node at (3.2, 5) {\tiny $P_{1,1}^+$};
\draw[blue] (0.8,0.59) --  (4.1,3.53);
\node at (2.35, 2.6) {\tiny $P_{1,r_1-1}^+$};
\draw[blue] (0.2,4.96) -- (3.8,11.49);
\node at (1.53, 8) {\tiny $P_{1,1}^-$};
\draw[blue] (0.2,8.96) -- (3.61,11.8);
\node at (1.4, 10.6) {\tiny $P_{1,r_0-1}^-$};

\draw[blue] (4.5,0.66) -- (8.1,11.33);
\node at (5.9, 6) {\tiny $P_{2,0}^{\pm}$};
\draw[blue] (4.6,0.63) -- (8.1,8.52);
\node at (6.9, 5) {\tiny $P_{2,1}^{+}$};
\draw[blue] (4.7,0.59) --  (8.1,5.52);
\node at (6.8, 3) {\tiny $P_{2,2}^{+}$};
\draw[blue] (4.8,0.51) -- (8.1,2.52);
\node at (6.8,1.2) {\tiny $P_{2,{r_2 -1}}^{+}$};

\draw[blue] (4.5,4.46) -- (7.8,11.49);
\node at (5.9,8.2) {\tiny $P_{2,1}^{-}$};

\draw[blue] (4.5,8.46) -- (7.61,11.8);
\node at (5.9,10.7) {\tiny $P_{2,{r_1 -1}}^{-}$};

\draw[blue] (12.5,0.66) -- (15.8,11.01);
\node at (13.8, 6) {\tiny $P_{n,0}^{\pm}$};

\draw[blue] (12.6,0.63) -- (15.8,8.02);
\node at (15.1, 5.5) {\tiny $P_{n,1}^{+}$};
\draw[blue] (12.7,0.59) --  (15.8,4.02);
\node at (15.1, 2.3) {\tiny $P_{n,{r_n-1}}^{+}$};

\draw[blue] (12.5,4.46) -- (15.5,11.15);
\node at (13.7,8) {\tiny $P_{n,1}^{-}$};
\draw[blue] (12.5,8.46) -- (15.15,11.5);
\node at (13.9,11) {\tiny $P_{n,r_{n-1}-1}^{-}$};

\node[vertex] at (4.8,11.49) {};
\node[vertex] at (8.8,11.49) {};
\node[vertex] at (12.8,11.49) {} ;

\node[vertex] at (3.8,0.51) {};
\node[vertex] at (7.8,0.51) {};
\node[vertex] at (11.8,0.51) {} ;
\node[vertex] at (15.15,0.51) {} ;

\node[vertex] at (3.8,4) {};
\node[vertex] at (4.8,4) {};
\node[vertex] at (3.8,8) {};
\node[vertex] at (4.8,8) {};

\node[vertex] at (7.8,3) {};
\node[vertex] at (8.8,3) {};
\node[vertex] at (7.8,6) {};
\node[vertex] at (8.8,6) {};
\node[vertex] at (7.8,9) {};
\node[vertex] at (8.8,9) {};

\node[vertex] at (11.8,4) {};
\node[vertex] at (12.8,4) {};
\node[vertex] at (11.8,8) {};
\node[vertex] at (12.8,8) {};

\end{tikzpicture}
    \caption{Generating objects in $\mathcal{W}(1;2,2,\ldots,2)$. Top-bottom and left-right are identified. }
\label{fig8}
\end{figure}

\begin{figure}[!htbp]
\begin{equation}
\begin{tikzpicture}[baseline=-0.8ex]
\matrix (m) [matrix of math nodes,row sep=1.6em,column sep=2em,minimum width=2em]
    {&P^-_{1,0} & P^-_{1,1}&\cdots&P^-_{1,r_0-1}&P^-_{1,r_0}\\
     &P^+_{1,0} & P^+_{1,1}&\cdots&P^+_{1,r_1-1}&P^+_{1,r_1}\\
     &S_{1,0}   & S_{1,1}  &\cdots&S_{1,r_1-1}  &           \\
    P^-_{2,r_1}&P^-_{2,r_1-1}&P^-_{2,r_1-2}&\cdots&P^-_{2,0}& \\
    P^+_{2,r_2}&P^+_{2,r_2-1}&P^+_{2,r_2-2}&\cdots&P^+_{2,0}& \\
    &S_{2,r_2 -1} &S_{2,r_2 -2} & \cdots & S_{2,0}& \\
     & \cdots & \cdots & \cdots   & \cdots & \\
     & \cdots & \cdots & \cdots   & \cdots & \\
    &S_{n-1,r_{n-1}-1} & S_{n-1,r_{n-1}-2} & \cdots &S_{n-1,0}& \\
    & P^-_{n,0} & P^-_{n,1} &\cdots&P^-_{n,r_{n-1}-1}&P^-_{n, r_{n-1}}\\
    & P^+_{n,0} & P^+_{n,1} &\cdots&P^+_{n,r_{n}-1}&P^+_{n,r_{n}} \\
     &S_{0,0} &  S_{0,1} &  \cdots & S_{0,r_0-1}& \\
    P^-_{1,{r_0}}& P^-_{1,r_0-1}&P^-_{1,r_0-2}&\cdots&P^-_{1,0}& \\};
     
     \path[-stealth]
       (m-1-2) edge node [left] {\scriptsize $=$} (m-2-2)
               edge node [above] {\scriptsize $x_{1,0}$} (m-1-3)
       (m-1-3) edge node [above] {\scriptsize $x_{1,1}$} (m-1-4)
       (m-1-4) edge node [above] {} (m-1-5)
       (m-1-5) edge node [above] {\scriptsize $x_{1,r_0-1}$} (m-1-6)
       
       (m-1-6) edge node [left] {\scriptsize $=$} (m-2-6)

       (m-2-2) edge node [above] {\scriptsize $y_{1,0}$} (m-2-3)
       (m-2-3) edge node [above] {\scriptsize $y_{1,1}$} (m-2-4)
       (m-2-4) edge node [above] {} (m-2-5)
       (m-2-5) edge node [above] {\scriptsize $y_{1,r_1-1}$} (m-2-6)

       (m-3-2) edge node [left] {\scriptsize $a_{1,0}$} (m-2-2)
       (m-3-3) edge node [left] {\scriptsize $a_{1,1}$} (m-2-3)
       (m-3-5) edge node [left] {\scriptsize $a_{1,r_1-1}$} (m-2-5)

       (m-3-2) edge node [left] {\scriptsize $b_{1,0}$} (m-4-2)
       (m-3-3) edge node [left] {\scriptsize $b_{1,1}$} (m-4-3)
       (m-3-5) edge node [left] {\scriptsize $b_{1,r_1-1}$} (m-4-5)

       (m-4-1) edge node [left] {\scriptsize $=$} (m-5-1)
           
       (m-4-2) edge node [above] {\scriptsize $x_{2,r_1-1}$} (m-4-1)
       (m-4-3) edge node [above] {\scriptsize $x_{2,r_1-2}$} (m-4-2)
       (m-4-4) edge node [above] {} (m-4-3)
       (m-4-5) edge node [above] {\scriptsize $x_{2,0}$} (m-4-4)
       
       (m-4-5) edge node [left] {\scriptsize $=$} (m-5-5)

       (m-5-2) edge node [above] {\scriptsize $y_{2,r_2-1}$} (m-5-1)
       (m-5-3) edge node [above] {\scriptsize $y_{2,r_2-2}$} (m-5-2)
       (m-5-4) edge node [above] {} (m-5-3)
       (m-5-5) edge node [above] {\scriptsize $y_{2,0}$} (m-5-4)

       (m-6-2) edge node [left] {\scriptsize $a_{2,r_2-1}$} (m-5-2)
       (m-6-3) edge node [left] {\scriptsize $a_{2,r_2-2}$} (m-5-3)
       (m-6-5) edge node [left] {\scriptsize $a_{2,0}$} (m-5-5)

       (m-6-2) edge node [left] {\scriptsize $b_{2,r_2-1}$} (m-7-2)
       (m-6-3) edge node [left] {\scriptsize $b_{2,r_2-2}$} (m-7-3)
       (m-6-5) edge node [left] {\scriptsize $b_{2,0}$} (m-7-5)

       (m-9-2) edge node [left] {\scriptsize $a_{n-1,r_{n-1}-1}$} (m-8-2)
       (m-9-3) edge node [left] {\scriptsize $a_{n-1,r_{n-1}-2}$} (m-8-3)
       (m-9-5) edge node [left] {\scriptsize $a_{n-1,0}$} (m-8-5)

       (m-9-2) edge node [left] {\scriptsize $b_{n-1,r_{n-1}-1}$} (m-10-2)
       (m-9-3) edge node [left] {\scriptsize $b_{n-1,r_{n-1}-2}$} (m-10-3)
       (m-9-5) edge node [left] {\scriptsize $b_{n-1,0}$} (m-10-5)

       (m-10-2) edge node [left] {\scriptsize $=$} (m-11-2)
              edge node [above] {\scriptsize $x_{n,0}$} (m-10-3)
       (m-10-3) edge node [above] {\scriptsize $x_{n,1}$} (m-10-4)
       (m-10-4) edge node [above] {} (m-10-5)
       (m-10-5) edge node [above] {\scriptsize $x_{n,r_n-1}$} (m-10-6)
       
       (m-10-6) edge node [left] {\scriptsize $=$} (m-11-6)

       (m-11-2) edge node [above] {\scriptsize $y_{n,0}$} (m-11-3)
       (m-11-3) edge node [above] {\scriptsize $y_{n,1}$} (m-11-4)
       (m-11-4) edge node [above] {} (m-11-5)
       (m-11-5) edge node [above] {\scriptsize $y_{n,r_n-1}$} (m-11-6)

       (m-12-2) edge node [left] {\scriptsize $a_{n,0}$} (m-11-2)
       (m-12-3) edge node [left] {\scriptsize $a_{n,1}$} (m-11-3)
       (m-12-5) edge node [left] {\scriptsize $a_{n,r_n-1}$} (m-11-5)

       (m-12-2) edge node [left] {\scriptsize $b_{n,0}$} (m-13-2)
       (m-12-3) edge node [left] {\scriptsize $b_{n,1}$} (m-13-3)
       (m-12-5) edge node [left] {\scriptsize $b_{n,r_n-1}$} (m-13-5)

       (m-13-2) edge node [above] {\scriptsize $x_{1,r_0-1}$} (m-13-1)
       (m-13-3) edge node [above] {\scriptsize $x_{1,r_0-2}$} (m-13-2)
       (m-13-4) edge node [above] {} (m-13-3)
       (m-13-5) edge node [above] {\scriptsize $x_{1,0}$} (m-13-4);

\end{tikzpicture}
\end{equation}
 \caption{Quiver describing $\mathcal{W}(1;2,2,\ldots,2)$ where the top and bottom rows
    should be identified according to the given labels.} 
    \label{fig9} 
\end{figure}

\noindent
{\it Proof of Theorem A: case $r_i\ge 1$ for all $i$}. The required equivalences are 
established by matching the exceptional collections and their endomorphism 
algebras: see Theorem \ref{B-side-thm} and Sections \ref{W0-sec} and \ref{W1-sec}.
Specifically, in Theorem \ref{B-side-thm} we set $j_i=0$ and $m_i=-1$ for all $i$.
Let us assume first that $\CC=C(r_0,\ldots,r_n;k_1,\ldots,k_{n-1})$.
We use the correspondence
\begin{align}
&P^-_{i,j} \longleftrightarrow \PP_i(j,-1),\\
&P^+_{i,m} \longleftrightarrow \PP_i(0,m-1),\\
&S_{i,j} \longleftrightarrow \SS_{q_i}\{-k_ij\}[-1],
\end{align}
to identify the endomorphism algebra of the exceptional collection of Theorem \ref{B-side-thm}
with the one for the marked surface constructed in Section \ref{W0-sec} using the permutations
$$\si_i(x)=-k_i\cdot x \mod r_i$$
of $\Z/r_i$.
We have
$$[\si_i,\tau_i](x)=x+k_i+1 \mod r_i,$$
which means that the cycle decomposition of $[\si_i,\tau_i]$ has $p_i=\gcd(k_i+1,r_i)$ cycles of length $d_i=r_i/p_i$.
It remains to use the formula for the genus from Proposition \ref{topology}.

The case $\CC=C(r_1,\ldots,r_n;k_1,\ldots,k_n)$ is considered similarly using the results of Section \ref{W1-sec}.
\ed

\begin{rmk} If we use other $(j_i,m_i)$ in Theorem \ref{B-side-thm} we get a homeomorphic surface. This follows from the fact
that the commutator $[\si,\tau]$ does not change if we replace $\si$ by $\si\tau^m$.
\end{rmk}

We will finish the proof of Theorem A in the case when either $r_0=0$ or $r_n=0$ in Section \ref{local-B-sec} after Proposition
\ref{loc-corr-prop}.

\section{Localization}
\label{local} 
\subsection{Localization on the A-side}\label{local-A-sec}

In \cite[Sec. 3.5]{HKK} it was proved that removing a marked point on a boundary component corresponds to
localization of the partially wrapped Fukaya category given by taking the quotient (in the derived
sense, cf. \cite{drinfeld}) by the subcategory generated by objects supported near the boundary marked point. The
latter subcategory is generated by a single object in this dimension and this object is
exceptional if and only if there is another marked point on the same boundary component.

In Section \ref{Fuk-sec} we computed some categories $\WW (g;m_1,\ldots,m_d)$ in terms of generating 
exceptional collections, starting from either a linear data $(r_0,r_1,\ldots, r_n)$ with $r_i \geq 1$ and
$(\sigma_{1},\ldots,\sigma_{n-1}) \in \mathfrak{S}_{r_1} \times \ldots \times
\mathfrak{S}_{r_{n-1}}$ or a circular
data $(r_1,\ldots, r_n)$ with $r_i \geq 1$ and $(\sigma_1,\ldots, \sigma_{n}) \in
\mathfrak{S}_{r_{1}} \times \ldots \times \mathfrak{S}_{r_{n}}$. 

We can now use localization to compute $\mathcal{W}(g;m'_1,\ldots,m'_d)$ for any $0 \leq m'_i \leq m_i$. To do this, we will identify the objects supported near each marked point in terms of our
generators. This is easily done by using the determination of $\mathcal{W}(0;m)$ given in Section
\ref{diskan} and the cosheaf property of wrapped Fukaya categories proved in \cite[Sec. 3.6]{HKK}. 

In the cases at hand, the cosheaf property gives functors from $\mathcal{W}(0;3)$, resp.
$\mathcal{W}(0;4)$, to the categories $\mathcal{W}(g;m_1,\ldots, m_d)$ corresponding to triangular
and rectangular, regions depicted in
Figure \ref{fig10} illustrating the case where $\sigma_i= \id$. The case of non-trivial $\sigma_i$ is
similarly covered with triangular and rectangular regions. Thus, using the twisted complex from Eq. \ref{twcx}, we can
identify the objects supported near each marked point in terms of our generators.  

In the case of linear data the $r_0$ and $r_n$ marked points on the 
distinguished boundary components give objects $E^-_{1,j}, j=0,\ldots, r_0-1$ and $E^+_{n,j},
j=0,\ldots, r_n-1$ supported near them. Using the functors from $\mathcal{W}(0;3)$, we conclude that
these are given by the complexes:
\begin{align}\label{E1j-eq}
    E_{1,j}^-\  &:\ P_{1,j}^-[2] \to P_{1,j+1}^-[1] \\ 
    E_{n,j}^+\  &:\ P_{n,j}^+[2] \to P_{n,j+1}^+[1]  
\end{align}
All other boundary points give objects $E_{i,j}^+$ and $E_{i+1,j}^-$ for $i=1,\ldots, n-1$ and $j=
0,\ldots, r_i-1$. Using the functors from $\mathcal{W}(0;4)$, these can be expressed as iterated cones:
\begin{align}
    E_{i,j}^{-}\ &:\  S_{i-1,\sigma_{i-1}^{-1}(r_{i-1}-j-1)}[3] \to P_{i,j}^-[2] \to P_{i,j+1}^-[1] \\
    E_{i,j}^{+}\ &:\  S_{i,j}[3] \to P_{i,j}^+[2] \to P_{i,j+1}^+[1] 
\end{align}

In the case of circular data we have a similar situation. The objects
supported near the marked points are labeled by $E^{\pm}_{i,j}$ for $i=1,\ldots,n$ and $j= 0\ldots,
r_{i}-1$, where $i$ is considered as an element in $\mathbb{Z}/n$. There are only functors from
$\mathcal{W}(0;4)$ and these give iterated cones as before:
\begin{align}
    E_{i,j}^{-}\ &:\  S_{i-1,\sigma_{i-1}^{-1}(r_{i-1}-j-1)}[3] \to P_{i,j}^-[2] \to P_{i,j+1}^-[1] \\
    \label{Eij+eq}
    E_{i,j}^{+}\ &:\  S_{i,j}[3] \to P_{i,j}^+[2] \to P_{i,j+1}^+[1] 
\end{align}

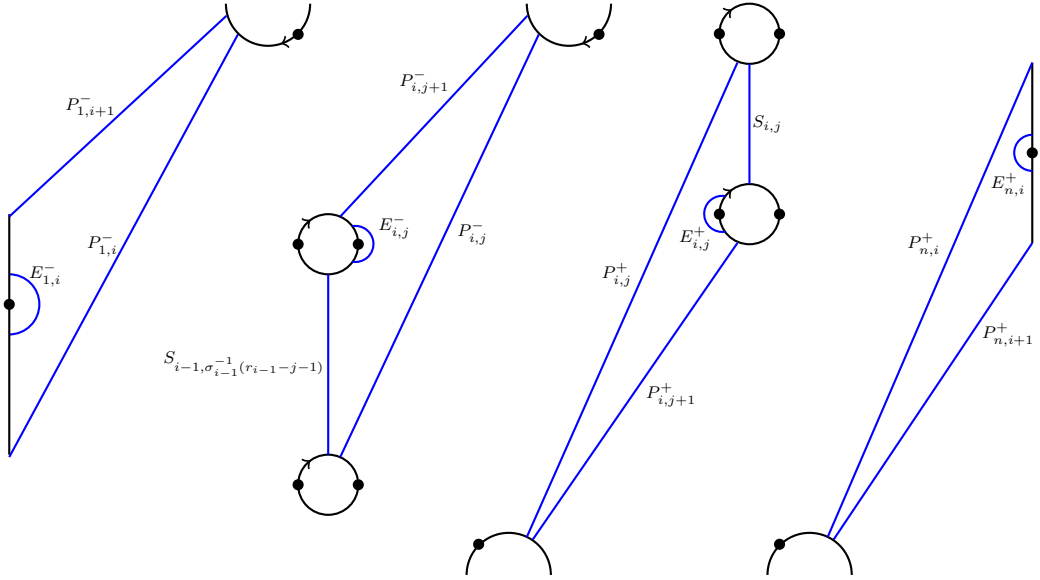
\begin{figure}[!htbp]
\centering
\begin{tikzpicture}[thick,scale=0.8, every node/.style={transform shape}]

\tikzset{vertex/.style = {style=circle,draw, fill,  minimum size = 2pt,inner        sep=1.5pt}}
    \tikzset{->-/.style={decoration={ markings,
        mark=at position #1 with {\arrow{>}}},postaction={decorate}}} 

\draw [blue, thick=1.5] (0, 7.5) arc (90:-90:0.5);
\node at (0.6, 7.5) {\tiny $E_{1,i}^-$};

\draw  [->-=.4, thick=1.5]((5,12) arc (0:-180:0.7);

\draw [thick=1.5] (0,8.5) -- (0,4.5);
\draw[blue] (0,4.46) -- (3.8,11.49);
\node at (1.55, 8) {\tiny $P_{1,i}^-$};
\draw[blue] (0,8.46) -- (3.61,11.8);
\node at (1.35, 10.3) {\tiny $P_{1,i+1}^-$};
\node[vertex] at  (0,7) {};
\node[vertex] at (4.8,11.49) {};

\begin{scope}[xshift=1cm]
    
\draw  [->-=.4, thick=1.5]((9,12) arc (0:-180:0.7);
\draw  [->-=.65, thick=1.5]((4.8,4) arc (360:0:0.5);
\draw  [->-=.65, thick=1.5]((4.8,8) arc (360:0:0.5);
\draw  [blue, thick=1.5]((4.7,8.3) arc (100:-100:0.3);
\node at (5.4, 8.3) {\tiny $E_{i,j}^-$};

\draw[blue] (4.3,4.5) -- (4.3,7.5);
    \node at (2.9, 6) {\tiny $S_{i-1,\sigma_{i-1}^{-1}(r_{i-1}-j-1)}$};

\draw[blue] (4.5,4.46) -- (7.8,11.49);
\node at (6.7,8.2) {\tiny $P_{i,j}^{-}$};

\draw[blue] (4.5,8.46) -- (7.61,11.8);
\node at (5.9,10.7) {\tiny $P_{i,j+1}^{-}$};

\node[vertex] at (8.8,11.49) {};

\node[vertex] at (3.8,4) {};
\node[vertex] at (4.8,4) {};
\node[vertex] at (3.8,8) {};
\node[vertex] at (4.8,8) {};

\end{scope}

\begin{scope}[xshift=4cm, yshift=2.5cm]

\draw  [thick=1.5]((5,0) arc (0:180:0.7);
\draw  [->-=.65,  thick=1.5]((8.8,6) arc (360:0:0.5);
\draw  [->-=.65, thick=1.5]((8.8,9) arc (360:0:0.5);
\draw  [blue, thick=1.5]((7.9,6.3) arc (80:280:0.3);
\node at (7.4, 5.6) {\tiny $E_{i,j}^+$};

\draw[blue] (8.3,6.5) -- (8.3,8.5);
\node at (8.6, 7.5) {\tiny $S_{i,j}$};
\draw[blue] (4.6,0.63) -- (8.1,8.52);
\node at (6.1, 5) {\tiny $P_{i,j}^{+}$};
\draw[blue] (4.7,0.59) --  (8.1,5.52);
\node at (7, 3) {\tiny $P_{i,j+1}^{+}$};
\node[vertex] at (3.8,0.51) {};
\node[vertex] at (7.8,6) {};
\node[vertex] at (8.8,6) {};
\node[vertex] at (7.8,9) {};
\node[vertex] at (8.8,9) {};

\end{scope}

\begin{scope}[xshift=1cm, yshift=2.5cm]

\draw  [thick=1.5]((13,0) arc (0:180:0.7);
\draw [thick=1.5] (16,5.52) -- (16,8.52);
\draw[blue] (12.6,0.63) -- (16,8.52);
\node at (14.2, 5.5) {\tiny $P_{n,i}^{+}$};
\draw[blue] (12.7,0.59) --  (16,5.52);
\node at (15.6, 4) {\tiny $P_{n,i+1}^{+}$};
\node[vertex] at (16,7.02) {};
\draw [blue, thick=1.5] (16, 7.32) arc (90:270:0.3);
\node at (15.6, 6.5) {\tiny $E_{n,i}^+$};

\node[vertex] at (11.8,0.51) {} ;

\end{scope}

\end{tikzpicture}
    \caption{Functors from $\mathcal{W}(0;3)$ and $\mathcal{W}(0;4)$ corresponding to embeddings of
    disks}

    \label{fig10}
\end{figure}

\subsection{Localization on the $B$-side}\label{local-B-sec}

For each node and (in the case of balloon chain) each smooth stacky point on $\bC$ we consider
some simple $\cA_{\bC}$-modules which turn out to be exceptional objects in the derived category.

Namely, for $i=1,\ldots,n$ and integer $j$, we have $\cA_{\bC}$-modules
$$\wt{\SS}^\pm_{i}(j)=
\left(\begin{matrix} \pi_{i*}\OO(jq_{i,\pm})|_{q_{i,\pm}}\\ \pi_{i*}\OO(jq_{i,\pm})|_{q_{i,\pm}} \end{matrix}\right),$$
which fit into exact sequences
\begin{equation}\label{tilde-S-def-eq}
\begin{split}
0\to \PP_i(j-1,m)\to \PP_i(j,m)\to \wt{\SS}^-_i(j)\to 0 \\
0\to \PP_i(j,m-1)\to \PP_i(j,m)\to \wt{\SS}^+_i(m)\to 0.
\end{split}
\end{equation}
Note that $\wt{\SS}^\pm_i(j)$ is supported at the point $\pi_i(q_{i,\pm})$. In the case when this point is not a node we set
$\EE^\pm_i(j)=\wt{\SS}^\pm_i(j)$. 

If $\pi_i(q_{i,\pm})$ is a node then we observe that there are natural inclusions 
$\SS_{q_i}\{-k_im\}\hra \wt{\SS}^+_{i}(m)$ and $\SS_{q_{i-1}}\{-j\}\hra \wt{\SS}^-_{i}(j)$.
Now we define the simple
$\cA_{\bC}$-module $\EE^\pm_i(j)$ as the corresponding quotient. Thus, we have exact sequences
$$0\to \SS_{q_i}\{-k_im\}\to \wt{\SS}^+_{i}(m)\to \EE^+_{i}(m)\to 0,$$
$$0\to \SS_{q_{i-1}}\{-j\}\to \wt{\SS}^-_{i}(j)\to \EE^-_{i}(j)\to 0.$$

We claim that $\EE^\pm_{i}(j)$ is an exceptional 
object precisely when this point is either a node or has a nontrivial stacky structure.

\begin{lem} Unless $\pi_i(q_{i,\pm})$ is a smooth point with trivial stacky structure,
the object $\EE^\pm_{i}(j)$ is exceptional.
\end{lem}

\Pf . To calculate morphisms involving $\EE^\pm_{i}(j)$ we can restrict to a formal neighborhood of the point
$\pi_i(q_{i,\pm})$. Also, tensoring with a line bundle of the form $\MM\{\ba\}$ we reduce to the case $j=0$.
Assume first that this point is a node $q_i$, so that a neighborhood of $q_i$ is isomorphic
to the stack quotient of $xy=0$ by $\mu_r$, where $r=r_i$.

Consider first the case when $r=1$. Then the completion $\hat{\cA}$ of $\cA$ at $q_i$ 
can be identified with the completion of the
path algebra of the following quiver with relations:
\begin{equation}\label{compl-quiver}
\xymatrix
{
\stackrel{-}\circ \ar@/^/[rr]^{u_{-}}  & &  \stackrel{0}\circ \ar@/^/[ll]^{v_{-}}
 \ar@/_/[rr]_{v_{+}}
 & &
\ar@/_/[ll]_{u_{+}} \stackrel{+}\circ}  \qquad v_{+} u_{-} = 0, \quad  v_{-} u_{+} = 0
\end{equation}
(see \cite[Rem.\ 2.7]{BD}). Furthermore, the simple $\hat{\cA}$-module $\SS_0:=\SS_{q_i}$ 
corresponds to the middle vertex, while $\SS^+:=\EE^+_{i}$ and $\SS^-:=\EE^-_{i+1}$ correspond to two other vertices.
The projective resolutions of $\SS_\pm$ have the form (see \cite[Rem.\ 2.7]{BD})
\begin{equation}\label{S+resolution-eq}
0\to \PP^\mp\rTo{j_\mp} \PP_0\rTo{i_\pm} \PP^\pm \to \SS^\pm\to 0,
\end{equation}
where $\PP^\pm$ (resp., $\PP_0$) is the projective cover of $\SS^\pm$ (resp., $\SS_0)$.
Computing $\Ext^*(\SS^\pm,\SS^\pm)$ using these resolutions we immediately deduce that $\SS^\pm$
are exceptional.

In the case of a node $q_i$ with $r=r_i>1$ the modules $\EE^+_{i}$ and $\EE^-_{i+1}$ correspond to the simple modules
$\SS^\pm$ on the formal neighborhood of a node in $xy=0$, viewed as $\mu_r$-equivariant $\hat{\cA}$-modules,
so the above computation can still be applied. 

Finally, in the case when $\pi_i(q_{i,\pm})$ is a smooth stacky point, the fact that 
$\EE^\pm_{i}(j)=\wt{\SS}^\pm_{i}(j)$ is exceptional follows immediately from the locally projective resolution
\eqref{tilde-S-def-eq}.
\ed

\begin{prop}\label{loc-corr-prop}
Under the equivalence of Theorem A (obtained using Theorem \ref{B-side-thm} with $j_i=0$, $m_i=-1$),
the $\cA_\bC$-modules $\EE^-_{i}(j)$, for $j=1,\ldots,r_{i-1}$ (resp., $\EE^+_i(j)$ for $j=0,\ldots,r_i-1$), correspond to 
the objects $E^-_{i,j-1}[-1]$ (resp., $E^+_{i,j}[-1]$) in the wrapped Fukaya category (see Sec.\ \ref{local-A-sec}).
\end{prop}

\Pf . Let $\Exc$ denote the direct sum of all the objects of our exceptional collection in $D^b(\cA_\bC-\mod)$.
We are going to describe the right modules $\Hom(\Exc,?)$ over the endomorphism algebra of our exceptional collection
associated with $\EE^-_{i}(j)$ (resp., $\EE^+_i(m)$). Assume first that our object is supported at a node $q=q_{i-1}$ 
(resp., $q=q_i$).
Note that as before, the computation
can be done locally near this node, so we can use $\mu_r$-equivariant modules (where $r=|\Aut(q)|$) over the completion of the Auslander order
at the node of $xy=0$, and our object is the simple object $\SS^+$ (resp., $\SS^-$) with some equivariant structure.
Thus, we get that the only nontrivial spaces of morphisms from modules
of the form $\PP_i(j',m')$ from our collection are 
\begin{itemize}
\item the $1$-dimensional space $\Hom(\PP_i(j,-1), \EE^-_i(j))$ (resp., $\Hom(\PP_i(0,m),\EE^+_i(m))$);
\item in the case $j=0$ (resp., $m=r-1$), the $1$-dimensional space 
$\Hom(\PP_i(0,m'),\EE^-_i(0))$ (resp., $\Hom(\PP_i(j',-1),\EE^+_i(r-1)$).
\end{itemize}
Also, we have a $1$-dimensional extension space
$$\Ext^1(\SS_q\{-j-1\},\EE^-_{i}(j)) \ \ 
(\text{resp., } \Ext^1(\SS_q\{-k_i(m+1)\},\EE^+_{i}(m))),$$
which comes from the locally projective resolution \eqref{Sq-ex-seq-U} of $\SS_q$ (as in the proof of Theorem \ref{B-side-thm},
we use tensoring with line bundles $\MM\{\ba\}$). Furthermore,
the generator of this $\Ext$-space is obtained as the composition of the natural morphisms
\begin{align*}
&\SS_q\{-j-1\}[-1]\rTo{b_{i-1}(j)} \PP_i(j,-1)\to \EE^-_i(j) \\ 
&(\text{resp., } \ \SS_q\{-k_i(m+1)\}[-1]\rTo{a_i(m+1)} \PP_i(0,m)\to \EE^+_i(m)).
\end{align*}
In the case $j=0$ (resp., $m=r-1$), we also have similar nonzero compositions of the $\Ext$-classes $b_{i-1}(0)$
(resp., $a_i(0)$) with the maps $\PP_i(0,m')\to \EE^-_i(0)$ (resp., $\PP_i(j',-1)\to \EE^+_i(r-1)$).

Thus, we see that the module $\Hom(\Exc,\EE^-_i(j))$ (resp., $\Hom(\Exc,\EE^+_i(m))$) is always concentrated in degree $0$.
In the case $j\neq 0$ (resp., $m\neq r-1$) it is generated by a single element
$$d_i(j)\in\Hom(\PP_i(j,-1),\EE^-_i(j)) \ \ (\text{resp., } c_i(m)\in\Hom(\PP_i(0,m),\EE^+_i(m))).$$
In the case $j=0$ (resp., $m=r-1$) the generator is
$$d_i(0)\in\Hom(\PP_i(0,r-1),\EE^-_i(0)) \ \ (\text{resp., } c_i(r-1)\in\Hom(\PP_i(r,-1),\EE^+_i(r-1))).$$
In either case the defining relation is that $dx=0$ (resp., $cy=0$) whenever the composition is possible.

In the case when our object is supported at a stacky point (which can happen when $\bC$ is a balloon chain)
there are no nonzero morphisms from objects of the form $\SS_q\{a\}$, so the module
$\Hom(\Exc,\EE^-_i(j))$ (resp., $\Hom(\Exc,\EE^+_i(m))$) is still generated by the same elements $d_i(j)$ (resp., $c_i(m)$)
as above, with the defining relations $dx=0$ and $db=0$ (resp., $cy=0$ and $ca=0$) whenever the composition
is possible.

Using the representations by complexes \eqref{E1j-eq}--\eqref{Eij+eq}
it is easy to compute the modules corresponding to the objects $E^\pm_{i,j}$ on the A-side. This gives the required matching.
\ed


\noindent
{\it Proof of Theorem A: case $r_0=0$ or $r_n=0$}.
Assume that $r_0=0$, so that $B(0,r_1)$ is the affine line with one stacky point.
In this case we can view $\CC=C(0,r_1,\ldots,r_n;k_1,\ldots,k_{n-1})$ as an open substack in
$\ov{\CC}:=C(1,r_1,\ldots,r_n;k_1,\ldots,k_{n-1})$, namely the complement to the point $q_-:=q_{1,-}\in B(1,r_1)\sub\ov{\CC}$.
Note that the object $\EE_1^-$ in this case is given by the module 
$\left(\begin{matrix} \OO_{q_-}\\ \OO_{q_-} \end{matrix}\right)$.
Since $\cA_{\ov{\CC}}$ is isomorphic near $q_-$ to the matrix algebra over $\OO$, it follows that the restriction functor
$$\cA_{\ov{\CC}}-\mod\to \cA_\CC-\mod$$
identifies $\cA_\CC-\mod$ with the quotient of $\cA_{\ov{\CC}}-\mod$ by the Serre subcategory generated by $\EE_1^-$.
Hence, by the main result of \cite{miyachi}, we have an equivalence of derived categories 
$$D^b(\cA_{\ov{\bC}}-\mod)/\lan \EE_1^-\ran\simeq D^b(\cA_{\bC}-\mod).$$
Using the behavior of the partially wrapped Fukaya categories upon deleting one marked point (see Sec.\ \ref{local-A-sec}) and Proposition \ref{loc-corr-prop}, 
we see that the equivalence of
$D^b(\cA_{\ov{\bC}}-\mod)$ with $\WW(g;1,(2d_1)^{p_1},\ldots,(2d_{n-1})^{p_{n-1}},r_n)$ implies an equivalence
of $D^b(\cA_{\bC}-\mod)$ with $\WW(g;0,(2d_1)^{p_1},\ldots,(2d_{n-1})^{p_{n-1}},r_n)$.

The case when $r_n=0$ is considered similarly.
\ed


Next, using the approach of \cite[Sec.\ 4]{BD},
we would like to prove the equivalence of the quotient category of $D^b(\cA_\bC-\mod)$ by all of the objects
$\EE^\pm_i(j)$, supported at the nodes, with $D^b\Coh(\bC)$.

Let us denote by $\TT\sub \cA_\bC-\mod$ the subcategory formed by
direct sums of all the objects $\EE^\pm_i(j)$ supported at the nodes.

\begin{prop}\label{DbCoh-prop} 
The subcategory $\TT\sub \cA_\bC-\mod$ is a Serre subcategory. The functor 
$$\cA_\bC-\mod\to\Coh\bC: M\mapsto\und{\Hom}_{\cA_\bC}(\FF_\bC,M)$$ 
is exact and identifies $\Coh\bC$ with the Serre quotient $\cA_\bC-\mod/\TT$. 
Similarly, the corresponding derived functor identifies $D^b(\Coh\bC)$ with the Verdier quotient of
$D^b(\cA_\bC-\mod)$ by the triangulated (equivalently, thick) subcategory generated by $\TT$.
\end{prop}

\Pf . The assertion about derived categories is a consequence of the assertion about abelian categories
(see \cite{miyachi}).
In the non-stacky case the assertion about abelian categories 
was proved in \cite[Thm.\ 4.8]{BD}. Using the identification of $\bC$ near a node
with the quotient of the non-stacky nodal curve by $\mu_r$, one can check that the same proofs goes through
in our case. Namely, as in the proof of \cite[Thm.\ 4.8]{BD}, first one constructs some adjoint functors, then reduces
the assertion to proving that some natural transformations are isomorphisms and then checks the last assertion
locally.
\ed

\section{Perfect derived categories}\label{Perf-sec}

\subsection{Perfect derived category on the B-side}

\begin{lem}\label{simple-orth-lem}
Let $A$ be the completion of a path algebra of a finite quiver $Q$ with relations.
Assume that $A$ is Noetherian and has finite cohomological dimension. For every vertex $v$ we denote
by $S_v$ (resp., $P_v$) the simple $A$-module at the vertex $v$ (resp., the projective $A$-module generated
by the idempotent in $A$ corresponding to $v$). Then for any subset $\Sigma$ of vertices of $Q$ one has
the equality of full triangulated subcategories in the bounded derived category of finitely generated $A$-modules, $D^b(A-\mod)$,
$${}^\perp\lan S_v \ |\ v\not\in\Sigma\ran=\lan \PP_v \ |\ v\in\Sigma\ran.$$
\end{lem}

\Pf .
Clearly we have $\Ext^*(\PP_v,\SS_w)=0$ for $v\in\Sigma$ and $w\not\in\Sigma$.
Conversely, let $M$ be a bounded complex in the left orthogonal of $\lan S_v \ |\ v\not\in\Sigma \ran$.
We will prove that $M$ is in $\lan \PP_v \ |\ v\in\Sigma \ran$ by induction on the length of $M$.
For the base of induction, let us assume that $M$ is an object of the abelian subcategory $A-\mod$.
Then the fact that $\Hom(M,\SS_v)=0$ for all $v\not\in\Sigma$
implies the existence of a surjection $P\to M$ with $P$ a direct sum
of finitely many $\PP_v$ with $v\in\Sigma$.
Let us consider an exact sequence 
$$0\to M'\to P\to M\to 0.$$
Then $M'$ is still in the left orthogonal $\lan S_v \ |\ v\not\in\Sigma\ran$ and has smaller projective dimension
than $M$. So, continuing in this way we deduce that $M\in\lan \PP_v \ |\ v\in\Sigma\ran$.
Now for the step of induction, assume that $M^\bullet$ is a complex $[M^a\to\ldots\to M^{b-1}\to M^b]$.
It is easy to see that the condition $\Ext^*(M,S)=0$ implies that
$\Hom(H^bM,S)=0$. Thus, there exists a surjection $P\to H^bM$, with $P$ a finite direct sum of $\PP_v$ with $v\in\Sigma$.
Let us lift it to a map $f:P\to M^b$ and extend $f$ to the chain map of complexes
of $A$-modules
\begin{center}
\begin{tikzpicture} 
\matrix (m) [matrix of math nodes,row sep=3em,column sep=2.8em,minimum width=2em,nodes={text
    height=1.75ex,text depth=0.25ex} ]
        {M^a & \cdots & M^{b-2}&M^{b-1}&M^b\\
        N^a & \cdots \ &N^{b-2}&N^{b-1}&N^b \\};
 \path[-stealth]
    (m-2-1) edge node [left] {\tiny $\id$} (m-1-1)
    (m-2-3) edge node [left] {\tiny $\id$} (m-1-3)
    (m-2-4) edge node [left] {} (m-1-4)
    (m-2-5) edge node [left] {\tiny $f$} (m-1-5)

    (m-1-1)  edge node [above] {} (m-1-2)
    (m-1-2) edge node [above] {} (m-1-3)
    (m-1-3) edge node [above] {} (m-1-4)
    (m-1-4) edge node [above] {} (m-1-5)

    (m-2-1) edge node [above] {} (m-2-2)
    (m-2-2) edge node [above] {} (m-2-3)
    (m-2-3) edge node [above] {} (m-2-4)
    (m-2-4) edge node [above] {} (m-2-5)

        (m-2-5) node {\qquad\qquad$=P$};
\end{tikzpicture}
\end{center} 
where the rightmost square is cartesian and $N^i=M^i$ for $i\le b-2$.
It is easy to see that the chain map $N^\bullet\to M^\bullet$ is a quasi-isomorphism. 
We have an exact sequence of complexes
$$0\to P[-n]\to N^\bullet\to \sigma_{\le b-1}N^\bullet\to 0,$$
where $\sigma_{\le b-1}N^\bullet=[N^a\to\ldots\to N^{b-1}]$ is a complex of length one less
than $M^\bullet$. From this exact sequence we derive that $\sigma_{\le b-1}N^\bullet$ is in
the left orthogonal of $\lan \SS^+,\SS^-\ran$. By the induction assumption, this implies that
it is in $\lan\PP_v \ |\ v\in\Sigma\ran$. Now the same exact sequence shows that $N^\bullet$ (and hence,
$M^\bullet$) is in $\lan\PP_v \ |\ v\in\Sigma\ran$.
\ed

\begin{lem}\label{compl-orthog-lem}
Let $\hat{\cA}$ be the completion of the Auslander order of the curve $xy=0$ at the node, which we
identify with the completed path algebra of the quiver \eqref{compl-quiver}. 

\noindent
(i) One has
$${}^\perp\lan \SS^-\ran=\lan \SS^+\ran^\perp, \ \ {}^\perp\lan \SS^+\ran=\lan \SS^-\ran^\perp.$$

\noindent
(ii) The following subcategories in $D^b(\hat{\cA}-\mod)$ coincide:
\begin{itemize}
\item the triangulated subcategory generated by $\PP_0$;
\item the left orthogonal of $\lan \SS^+,\SS^-\ran$;
\item the right orthogonal of $\lan \SS^+,\SS^-\ran$.
\end{itemize}
\end{lem}

\Pf . (i) By the symmetry of the quiver \eqref{compl-quiver}, it is enough to prove the first equality.
By Lemma \ref{simple-orth-lem}, we have 
$${}^\perp\lan \SS^-\ran=\lan \PP_0,\PP^+\ran.$$
It remains to prove the equality
$$\lan \SS^+\ran^\perp=\lan \PP_0,\PP^+\ran.$$
Calculating using the projective resolution \eqref{S+resolution-eq}
one can easily check that $\Ext^*(\SS^+,\PP_0)=\Ext^*(\SS^+,\PP^+)=0$.
To show that $\lan \SS^+\ran^\perp$ is generated by $\PP_0$ and $\PP^+$ we use the
left adjoint functor $\la:D^b(\hat{\cA}-\mod)\to \lan \SS^+\ran^\perp$ to the inclusion.
It is enough to check that the image of any projective module under $\la$ is in $\lan \PP_0,\PP^+\ran$.
We have $\la(\PP_0)=\PP_0$, $\la(\PP^+)=\PP^+$, so it remains to calculate $\la(\PP^-)$.
The resolution \eqref{S+resolution-eq} shows that the space $\Hom^*(\SS^+,\PP^-)$ is $1$-dimensional and
is concentrated in degree $2$. Furthermore, from the same projective resolution we see that $\la(\PP^-)$
is represented by the complex $[\PP_0\to \PP^+]$, which is in $\lan \PP_0,\PP^+\ran$.

\noindent
(ii) By part (i), it is enough to prove the assertion about the left orthogonal of $\lan \SS^+,\SS^-\ran$.
Since the algebra $\hat{\cA}$ is Noetherian and has finite cohomological dimension
(see \cite[Sec.\ 2]{BD}), the required equality follows from Lemma \ref{simple-orth-lem}.
\ed

Let us now return to the setup of Section \ref{Aus-st-sec} and
consider the functor
\begin{equation}\label{Perf-functor-eq}
\Perf(\bC)\to D^b(\cA_\bC-\mod):G\mapsto \FF_\bC\ot_{\OO_\bC} G.
\end{equation}
Recall that we denote by $\TT\sub \cA_\bC-\mod$ the subcategory formed by direct sums of all the objects
$\EE_i^\pm(j)$ supported at the nodes.
In the non-stacky case the following result is essentially \cite[Prop.\ 2.8]{BD}.

\begin{prop}\label{Perf-char-B-prop} 
(i) The functor \eqref{Perf-functor-eq} is fully faithful. Its essential image is the subcategory 
$$\TT^\perp={}^\perp\TT\sub D^b(\cA_\bC-\mod)$$
consisting of all objects right (resp., left) orthogonal to all objects in $\TT$.

(ii) Assume that $\CC=C(r_0,r_1,\ldots,r_n;k_1,\ldots,k_{n-1})$ where all $r_i>0$ and either $r_0=1$ or $r_n=1$.
Define $Z\sub \bC$ by 
$$Z=\begin{cases} \{q_{1,-}\}, & \text{ if } r_0=1, r_n>1;\\ \{q_{n,+}\}, & \text{ if } r_0>1,r_n=1;\\
\{q_{1,-},q_{n,+}\}, & \text{ if } r_0=r_{n-1}=1.\end{cases}$$ 
Let $\ov{\TT}\sub D^b(\cA_\bC-\mod$ be the triangulated subcategory generated by $\TT$ and by those of
the objects $(\EE_1^-,\EE_n^+)$ that are supported at $Z$. Then the functor \eqref{Perf-functor-eq} induces an equivalence of
$$\ov{\TT}^\perp={}^\perp\ov{\TT}\sub D^b(\cA_\bC-\mod)$$
with the compactly supported perfect derived category $\Perf_c(\bC\setminus Z)$. 
\end{prop}

\Pf . (i) Lemma \ref{compl-orthog-lem} implies that an object $M\in D^b(\cA_\bC-\mod)$ belongs to $\TT^\perp$
(resp., ${}^\perp\TT$) if and only if for every node $q$,
the object $\hat{M}_q$, viewed as a $\mu_r$-equivariant $\hat{\cA}$-module
(where $\hat{\cA}$ is the completion of the Auslander order of the curve $xy=0$ at the node), after forgetting the $\mu_r$-equivariant structure, belongs to the subcategory
generated by $\PP_0$. The rest of the proof is similar to that of \cite[Prop.\ 2.8]{BD}.

\noindent
(ii) If $r_0=1$ then $\cA_\bC$ is isomorphic to the matrix algebra near $q_{1,-}$, and 
$\EE_1^-$ is an $\cA_\bC$-module corresponding to $\OO_{q_{1,-}}$. This easily implies that an object 
$F\in \Perf(\CC)\sub D^b(\cA_\bC-\mod)$ is left or right orthogonal to $\EE_1^-$ if and only if its support does not contain 
$q_{1,-}$. Since the support is closed, this is equivalent to the condition that $F$ belongs to the essential image
of the natural fully faithful embedding
$$\Perf_c(\CC\setminus\{q_{1,-}\})\hra \Perf(\CC).$$
The cases when $r_n=1$ or $r_0=r_n=1$ are considered similarly.
\ed

\subsection{Characterization on the A-side}

Under the equivalence of Theorem A, the triangulated subcategory $\TT \sub D^b(\cA_\bC-\mod)$ for $\bC =
C(r_0,r_1,\ldots, r_n;k_1,\ldots,k_{n-1})$ (resp. $\bC= R(r_1,\ldots,r_n;k_1,\ldots,k_n)$)
corresponds to the subcategory of $\mathcal{W}(g;m_1,\ldots,m_d)$ generated by the objects $E_{i,j}^+$ for $i=1,\ldots, n-1$ and $E_{i,j}^-$ for $i=2,
\ldots, n$ (resp. $E_{i,j}^{\pm}$
for $i=1,\ldots n$ and $j=0,\ldots r_i-1$.). 

We next give a nice characterization of the subcategory $\mathcal{\TT}^\perp = {}^\perp \TT$ as
a triangulated subcategory of
$\mathcal{W}(g;m_1,\ldots,m_d)$. Recall that by the geometricity result of Haiden-Katzarkov-Kontsevich \cite[Thm. 4.3]{HKK}, every indecomposable
object in $ \mathcal{W}(g; m_1,\dots, m_d)$ is represented by an admissible
Lagrangian (with a local system). Let $\TT_i$ be the subcategory of $\mathcal{W}(g;m_1,\ldots,m_d)$
generated by the $i$ Lagrangians
supported near the marked points at the $i^{th}$ boundary component, see Figure \ref{Fpic} for
the case $i=2$. 

\begin{figure}[!htbp]
\centering
\begin{tikzpicture}[thick,scale=1.2, every node/.style={transform shape}]

\tikzset{vertex/.style = {style=circle,draw, fill,  minimum size = 2pt,inner        sep=1.5pt}}
    \tikzset{->-/.style={decoration={ markings,
        mark=at position #1 with {\arrow{>}}},postaction={decorate}}}

\draw  [->-=.65,  thick=1.5]((8.8,6) arc (360:0:0.5);
\draw  [blue, thick=1.5]((7.9,6.3) arc (80:280:0.3);
\draw  [blue, thick=1.5]((8.7,6.3) arc (100:-100:0.3);
\node at (9.3, 6.3) {\tiny $E^-$};

\node at (7.4, 5.6) {\tiny $E^+$};
\node[vertex] at (7.8,6) {};
\node[vertex] at (8.8,6) {};

\end{tikzpicture}
    \caption{Objects generating $\TT_i$}
    \label{Fpic}
\end{figure}
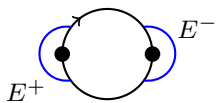

\begin{prop}\label{F-prop} The triangulated subcategory of $\mathcal{W}(g;m_1,\ldots,m_d)$ given by objects corresponding to Lagrangians (with local systems)
that do not end on the $i^{th}$ boundary component (where there are $m_i$ marked points) coincides with $\TT_i^\perp = {}^\perp \TT_i$.     
\end{prop}
\begin{proof} For simplicity of exposition we assume $m_i=2$, but the general argument is very
    similar for any $m_i>0$. Let us denote by $\hom(\cdot,\cdot)$ the 
  morphism complexes in the partially wrapped Fukaya category and by $\Hom(\cdot,\cdot)$ their cohomology.
  It suffices to prove that a geometrically represented indecomposable object $L$ of $
    \mathcal{W}(g;m_0,\ldots, m_d)$ is in $\TT_i^\perp = {}^\perp \TT_i$ if and only
    if $L$ is either compact or if does not have ends on the $i^{th}$ boundary component. Recall that the subcategory $\TT_i$ is generated by the objects $E^{\pm}$ supported near the marked points at the boundary components with two marked points, 
    By choosing the representatives for $E^{\pm}$ sufficiently near the marked points, we can ensure that they are disjoint from a given object
    $L$. Thus, if $L$ is compact or does not end at the boundary component near which
    $E^{\pm}$ is situated, then $\hom(L, E^{\pm}) = \hom(E^{\pm}, L)=0$. 
    Now suppose $L$ ends at the boundary component near which $E^{\pm}$ is
    supported. This boundary component has 2 marked points, let us distinguish the two components
    in the complement of these 2 marked points. Now, if precisely one of the end
    points of $L$ lies on one of these boundary components, then either both
    $\hom(E^+ , L)$ and  $\hom(L,E^-)$ are of rank 1 or both $\hom(L,E^+)$ and
    $\hom(E^-,L)$ are of rank 1, because in either case the chain complexes are of rank 1. In
    the case both of the end points of $L$ lie on the same boundary component, say between
    $E^+$ and $E^-$ along the orientation of the flow, we have morphisms as
    follows (see Figure \ref{illust}) :
\begin{equation}
\xymatrix
{
E^+ \ar@/^/[rr]^{a}  
\ar@/_/[rr]_{b}
 & &  L 
 \ar@/_/[rr]_{y}
\ar@/^/[rr]^{x} & &  E^-  }  \qquad ya = 0, \quad  xb = 0, \quad xa= yb.
\end{equation}
Thus, the chain complexes $\hom(L,E^\pm)$ are of rank 2. We claim that in fact the differential on either of these
complexes is zero. We will
    show this by passing to a cover. 
     
    Since $L$ is assumed to be a non-zero object, it cannot be
    represented by a boundary parallel curve, hence there exists a cover $\tilde{X}$ of
    $X$ such that we can find lifts $\tilde{E}^\pm$ and $\tilde{L}$, and such that only one end
    of $\tilde{L}$ lies in the region between $\tilde{E}^+$ and $\tilde{E}^-$ as illustrated in
    Figure \ref{illust}. The morphisms between these lifts are as follows:
\begin{equation}
\xymatrix
{
    \tilde{E}^+ \ar@/^/[rr]^{\tilde{a}}  
    & &  \tilde{L} 
    \ar@/^/[rr]^{\tilde{x}} & &  \tilde{E}^-  } 
\end{equation}
The covering map gives a functor
    \[ \mathcal{W}(\tilde{X}, \tilde{\Lambda}) \to \mathcal{W}(X, \Lambda) \]
    sending $\tilde{E}^\pm \to E^\pm$ and $\tilde{L} \to L$, and it induces an
    isomorphism of rank 1 modules
    \[ \hom(\tilde{E}^+ , \tilde{E}^-) \cong \hom(E^+, E^-)\]
by our choice of lifts of $E^\pm$. 

Finally, we note that there exists a non-trivial product map
    \[ \Hom(\tilde{L}, \tilde{E}^-) \otimes \Hom(\tilde{E}^+,\tilde{L}) \to
    \Hom(\tilde{E}^+ , \tilde{E}^-) \]
    given by $(\tilde{x} , \tilde{a}) \to \tilde{x} \tilde{a}$, which is mapped to a non-trivial product: 
   \[ \Hom(L, E^-) \otimes \Hom(E^+,L) \to
    \Hom(E^+ , E^-) \]
    Hence, it follows that the modules $\Hom(L,E^-)$ and $\Hom(E^+,L)$ are non-trivial, as
    required.
\end{proof}

\begin{figure}[!htbp]
\centering
\begin{tikzpicture}[thick,scale=1.2, every node/.style={transform shape}]

\tikzset{vertex/.style = {style=circle,draw, fill,  minimum size = 2pt,inner        sep=1.5pt}}
    \tikzset{->-/.style={decoration={ markings,
        mark=at position #1 with {\arrow{>}}},postaction={decorate}}}

\draw  [->-=.65,  thick=1.5]((8.8,6) arc (360:0:0.5);
\draw  [blue, thick=1.5](7.9,6.3) arc (80:280:0.3);
\draw  [blue, thick=1.5](8.7,6.3) arc (100:-100:0.3);
\node at (9.3, 6.3) {\tiny $E^-$};

\node at (7.4, 5.6) {\tiny $E^+$};
\node[vertex] at (7.8,6) {};
\node[vertex] at (8.8,6) {};

\draw  [->-=.65,  thick=1.5]((9.8,8.5) arc (360:0:0.5);

\draw[blue, thick=1.5] (8.1,6.45) to[in=210] (9, 9.5); 
\draw[blue, thick=1.5] (8.5,6.45) to[in=330] (10, 9.5); 

\draw[blue,thick=1.5] (9, 9.5) to[out=30, in=150] (10,9.5);

\node at (8.2, 8.2) {\tiny $L$};

    \begin{scope}[xshift=-7cm]

\draw  [->-=.65,  thick=1.5]((8.8,6) arc (360:0:0.5);
\draw  [blue, thick=1.5](7.9,6.3) arc (80:280:0.3);
\draw  [blue, thick=1.5](8.7,6.3) arc (100:-100:0.3);
        \node at (9.3, 6.3) {\tiny $\tilde{E}^-$};

        \node at (7.4, 5.6) {\tiny $\tilde{E}^+$};
\node[vertex] at (7.8,6) {};
\node[vertex] at (8.8,6) {};

\draw  [->-=.65,  thick=1.5](9.3,8) arc (240:80:0.5);
\draw  [thick=1.5,dashed ](9.65,8.93) arc (80:50:0.5);
\draw  [thick=1.5,dashed](9.3,8) arc (240:280:0.5);

\draw[blue, thick=1.5] (8.1,6.45) to[in=210] (9, 10.5); 

        \node at (8.2, 8.2) {\tiny $\tilde{L}$};

\begin{scope}[xshift=3cm, yshift=3cm]
\draw  [->-=.65,  thick=1.5]((8.8,6) arc (360:0:0.5);

\draw[blue, thick=1.5] (8.5,6.45) to[in=30] (6, 7.5); 

\node[vertex] at (7.8,6) {};
\node[vertex] at (8.8,6) {};

\end{scope}

    \end{scope}

\draw[->-=1, thick=1.5](5.5,7.3) -- (7,7.3);

\end{tikzpicture}
    \caption{An illustration of covering}
    \label{illust}
\end{figure}
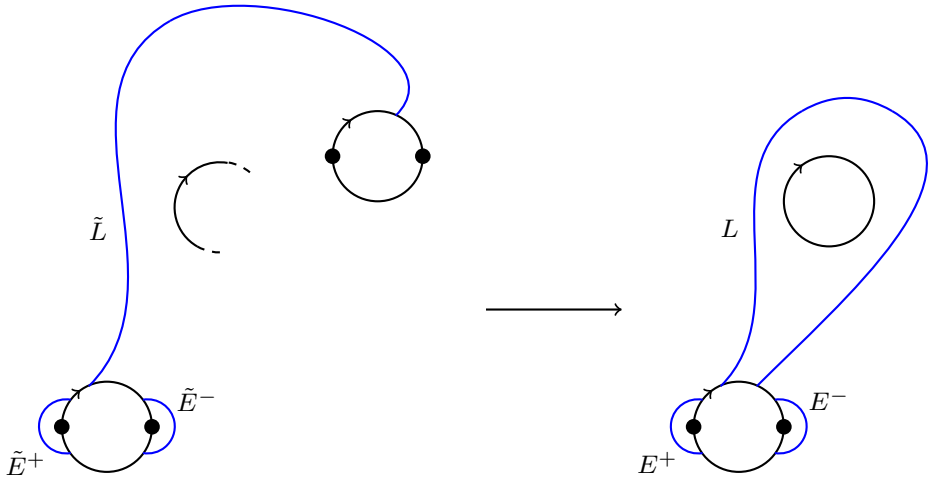

By repeatedly applying Prop. \ref{F-prop} we get the following result.

\begin{cor} \label{F-01-prop} (i) In the case of linear data, let $\TT$ be the subcategory of $
    \mathcal{W}(g;m_1,\ldots, m_d)$  generated by the objects $E_{i,j}^+$ for $i=1,\ldots, n-1$ and $E_{i,j}^-$
    for $i=2, \ldots, n$. Assume that all $r_i>0$. Then the subcategory
    $ \mathcal{F}(g;r_0, (0)^{r_1+\ldots r_{n-1}},r_n)\sub  \mathcal{W}(g;m_1,\ldots m_d)$
    coincides with $\TT^\perp = {}^\perp \TT$.     

(ii) In the case of circular data, let $\TT$ be the subcategory of $ \mathcal{W}(g;m_1,\ldots, m_d)$  generated by the objects $E_{i,j}^{\pm}$
for $i=1,\ldots n$ and $j=0,\ldots r_i-1$. Then the subcategory
$ \mathcal{F}(g;(0)^{r_1+\ldots r_n})\sub  \mathcal{W}(g;m_1,\ldots, m_d)$
coincides with $\TT^\perp = {}^\perp \TT$.
\end{cor}

\subsection{Proof of Theorem B}

Assume first that all $r_i>0$.
By Proposition \ref{loc-corr-prop}, the image of the subcategory $\TT\sub D^b(\cA_\bC-\mod)$
under the equivalence of Theorem A consists of Lagrangians supported near the interior boundary components.
Now Proposition \ref{Perf-char-B-prop}(i) and Corollary \ref{F-01-prop} 
imply that the image of $\Perf(\bC)$, embedded into $D^b(\cA_\bC-\mod)$
via \eqref{Perf-functor-eq}, corresponds under the equivalence of Theorem A precisely to
$\FF(g;r_0,(0)^{p_1+\ldots+p_{n-1}},r_n)$ in the case when $\bC=C(r_0,\ldots,r_n;k_1,\ldots,k_{n-1})$
(resp., $\FF(g;(0)^{p_1+\ldots+p_n})$ in the case when $\bC=R(r_1,\ldots,r_n;k_1,\ldots,k_n)$).

In the case when $r_0=0$ and $r_n>0$ we use the characterization of the embedding 
$$\FF(g;0,(0)^{p_1+\ldots+p_{n-1}},r_n)\hra \WW(g;1,(2d_1)^{p_1},\ldots,(2d_{n-1})^{p_{n-1}},r_n)$$
and Proposition \ref{Perf-char-B-prop}(ii). If $r_0=r_n=0$ then we use the embedding
$$\FF(g;0,(0)^{p_1+\ldots+p_{n-1}},0)\hra \WW(g;1,(2d_1)^{p_1},\ldots,(2d_{n-1})^{p_{n-1}},1).$$

The equivalences involving $D^b\Coh(\bC)$ follow from Proposition \ref{DbCoh-prop} and
the corresponding fact about partially wrapped Fukaya categories (see Sec.\ \ref{local-A-sec}).
\ed

\subsection{Dualities}
\label{dualities}

It is known that for a scheme $Y$, proper over a field $\k$,
one has the duality equivalences (see \cite{BzNP}):
\begin{align*} \mathrm{Perf}(Y) \simeq \mathrm{Fun^{ex}}(D^b
    \mathrm{Coh}(Y)^{op},
    \mathrm{Perf}\, \k)\\     
    D^{b} \mathrm{Coh}(Y) \simeq \mathrm{Fun^{ex}}(\mathrm{Perf}(Y)^{op}, \mathrm{Perf}\, \k), 
    \end{align*} 
where $\mathrm{Fun}^{ex}$ stands for DG-category of exact functors.

Thus, the homological mirror symmetry equivalences  
\eqref{Tn-mirror-eq}, \eqref{Tn-mirror-bis-eq} for $T_n$
imply the same duality between $D^{b} \mathcal{F}(T_n)$ and
$\mathcal{W}(T_n)$ (where we take Fukaya categories with coefficients in $\k$).

For a general Weinstein domain $X$, one expects to have an equivalence:
\[ \mathcal{F}(X) \simeq \mathrm{Fun^{ex}}( \mathcal{W}(X)^{op},
\mathrm{Perf}\, \k )\] 
The analogue of this statement in the world of microlocal sheaves is known \cite{nadler}. Also, a weaker
but in many cases equivalent statement was proved in \cite[Thm. 4]{EkLe}. 

On the other hand, the full duality statement is false in general, i.e., one cannot always recover
$ \mathcal{W}(X)$ from $D^{b} \mathcal{F}(X)$. For example, this is the case
when $X = T^*M$ and $M$ is not simply-connected. 

More generally, one expects the following duality (cf. \cite{nadler}):
\[  \mathcal{F}(X,\Lambda) \simeq \mathrm{Fun}^{ex}(
\mathcal{W}(X,\Lambda)^{op}, \mathrm{Perf}\,
\k ). \]
We can prove such duality for the categories considered in this paper.

\begin{prop}\label{duality-prop} There is a natural quasi-equivalence
$$\FF(g;m_1,\ldots,m_d)\rTo{\sim} \mathrm{Fun^{ex}}(\WW(g;m_1,\ldots,m_d)^{op},\mathrm{Perf}\, \k).$$
\end{prop}

\Pf . In the case when all $m_i$ are positive we have
$\FF(g;m_1,\ldots,m_d)=\WW(g;m_1,\ldots,m_d)$ and this category is smooth and proper (see \cite[Prop.\ 3.4]{HKK}), 
which implies
the needed self-duality (see \cite[Sec.\ 5.4]{Toen}).

Now suppose that $m_1=\ldots=m_r=0$ and $m_i>0$ for $i>r$. 
Let us set for brevity $\FF:=\FF(g;m_1,\ldots,m_d)$, $\WW:=\WW(g;m_1,\ldots,m_d)$.
By Proposition \ref{F-prop}, we can identify $\FF$
with $\TT^\perp$ in $\wt{\WW}:=\WW(g;(2)^r,m_{r+1},\ldots,m_d)$, where $\TT$
is generated by objects supported near the marked points of the first $r$ boundary components.
On the other hand, we have an equivalence
$$\WW\simeq \wt{\WW}/\TT$$
(see Section \ref{local-A-sec}). Hence, by the property of dg-quotients (see \cite[Thm.\ 1.6.2]{drinfeld}), we have
a quasi-equivalence of $\mathrm{Fun}^{ex}(\WW,\mathrm{Perf}\, \k)$ with the full subcategory of
$\mathrm{Fun}^{ex}(\wt{\WW},\mathrm{Perf}\, \k)$ consisting of the functors annihilating $\TT$.
But by \cite[Prop. 3.4]{HKK}, $\wt{\WW}$ is smooth and proper, so we can identify the latter subcategory with $\TT^{\perp}$,
hence, with $\FF$.
\ed

\subsection{Categorical resolutions of $\bC$}

\begin{prop}\label{nc-res-prop} Let $\TT(\Sigma)\sub \TT\sub D^b(\cA_\bC-\mod)$ be the triangulated subcategory generated
by a collection $\Sigma$ of objects $\EE_i^\pm(j)$ (where $j\in\Z/r_{i,\pm}$), supported at the nodes. 
Assume that for every $(i,j)$, the set $\Sigma$ contains at most one of the objects $(\EE^+_i(j), \EE^-_{i+1}(j'))$,
where $j'\equiv k_ij-1 \mod r_i$.
Then the subcategory $\TT(\Sigma)$ is admissible, and the composed functor
\begin{equation}\label{nc-res-fun}
\Perf(\bC)\to D^b(\cA_\bC-\mod)\to D^b(\cA_\bC-\mod)/\TT(\Sigma)
\end{equation}
is fully faithful.
\end{prop}

\Pf . We claim that a collection $\Sigma$ can be ordered, so that it is exceptional.
Indeed, this follows immediately from the fact 
that the only possibly nontrivial morphism spaces between the objects $(\EE^\pm_i(j))$,
supported at the nodes, are 
$$\Hom^*(\EE^+_i(j),\EE^-_{i+1}(j')) \ \text{ and }\ \Hom^*(\EE^-_{i+1}(j'),\EE^+_i(j))$$
for $-k_ij+j'\equiv -1 \mod r_i$.
This can be immediately seen on the A-side using Proposition \ref{loc-corr-prop}, or proved on the B-side using the description
of the completion of $\hat{\cA}_q$ at a node in terms of the quiver \eqref{compl-quiver}.

Thus, the subcategory $\TT(\Sigma)$, generated by $\Sigma$, is admissible, and we have a semiorthogonal decomposition
$$D^b(\cA_\bC-\mod)=\lan \TT(\Sigma)^\perp, \TT(\Sigma)\ran.$$
As we have seen in Proposition \ref{Perf-char-B-prop}, the functor {\eqref{Perf-functor-eq} factors through
$\TT(\Sigma)^\perp$. This implies that the functor \eqref{nc-res-fun} is fully faithful.
\ed

Note that the functors of the form \eqref{nc-res-fun} can be viewed as categorical resolutions of stacky curves $\bC$,
since the corresponding categories $D^b(\cA_\bC-\mod)/\TT(\Sigma)\simeq\TT(\Sigma)^\perp$ are smooth.

\begin{example} Let us consider the case when $\bC=C$ is the irreducible nodal curve of arithmetic genus $1$.
In this case the exceptional collection of \cite{BD} gives an equivalence of the category $D^b(\cA_C-\mod)$ 
with the derived category of finite-dimensional representations $D^b(Q)$ of the following quiver $Q$ with relations,
$$
\xymatrix
{
\stackrel{1}\circ \ar@/^/[rr]^{a}  
\ar@/_/[rr]_{b}
 & &  \stackrel{2}\circ 
 \ar@/_/[rr]_{y}
\ar@/^/[rr]^{x} & & \stackrel{3}\circ}  \qquad ya = 0, \quad  xb = 0.
$$
As we know, this category is also equivalent to $\WW(1;2)$.
The two exceptional objects $\EE^\pm$ correspond to representations
$$
\xymatrix
{
\stackrel{\C}\circ \ar@/^/[rr]^{1}  
\ar@/_/[rr]_{0}
 & &  \stackrel{\C}\circ 
 \ar@/_/[rr]_{0}
\ar@/^/[rr]^{1} & & \stackrel{\C}\circ}  \qquad 
\xymatrix
{
\stackrel{\C}\circ \ar@/^/[rr]^{0}  
\ar@/_/[rr]_{1}
 & &  \stackrel{\C}\circ 
 \ar@/_/[rr]_{1}
\ar@/^/[rr]^{0} & & \stackrel{\C}\circ} 
$$
The quotient of $D^b(Q)$ by either $\EE^+$ or $\EE^-$, which is equivalent to $\WW(1;1)$,
is a well-known categorical resolution of $C$ (see \cite[Sec.\ 3.5]{Kuz}). It is not hard to describe 
$D^b(Q)/\lan\EE^+\ran$ more explicitly. Namely, we can identify it with the subcategory $\lan\EE^+\ran^\perp$
in $D^b(Q)$ and take as generators the objects $M_2=[P_2\xrightarrow{\ b\ }P_1]$ and
    $M_1=[P_3\xrightarrow{\ y\ }P_2]$,
where $P_i$ is the projective $Q$-representation corresponding to the vertex $i$
(note that $M_1$ is quasi-isomorphic to an actual $Q$-representation, while $M_2$ is a complex with nontrivial $H^{-1}$ and $H^0$). It is easy to check that algebra $\Ext^*(M_1\oplus M_2,M_1\oplus M_2)$
is isomorphic to the algebra of the following graded quiver with relations:

\begin{figure}[htb!]
\centering
\begin{tikzpicture} [scale=1.3]

\tikzset{->-/.style={decoration={ markings,
            mark=at position #1 with {\arrow[scale=2,>=stealth]{>}}},
            postaction={decorate}}}
\node at (-0.2,0.1) {$\circ$};
\node at (2.2,0.1) {$\circ$};

    \draw [->-=0.999]  (0,0.25) to[in=130,out=50] (2,0.25);
    \draw [->-=1]   (0,0.1) to[in=165,out=15] (2,0.1);
    \draw [->-=1]  (2,-0.05) to[in=335,out=195] (0,-0.05);

    \node at (0.9,0.8) {$a$};
    \node at (0.9,0.4) {$c$};
    \node at (0.9,-0.4) {$b$};

    \node at (4,0.1) {$ab =bc =0.$}; 
\end{tikzpicture}
\end{figure}
Here $\deg(a)=\deg(c)=0$ and $\deg(b)=1$. 
Furthermore, one can easily calculate the Hochschild cohomology of this algebra (e.g., using Bardzell's resolution 
\cite{Bardzell}) and deduce that it is intrinsically formal. Hence, it indeed describes $D^b(Q)/\lan\EE^+\ran\simeq
\WW(1;1)$.

We note that this algebra plays an important role in bordered Heegaard Floer theory
\cite[Sec. 11.1]{LOT}: it appears as the algebra associated to the punctured torus (see also
\cite{HRW}). The relation of the Fukaya categories of surfaces and, more generally, of symmetric powers of surfaces to Heegaard Floer theory was discovered in the work \cite{LePe}. The specific relation to partially wrapped Fukaya categories was elucidated by Auroux in \cite{auroux}. 

\end{example}

\end{document}